\documentclass{amsproc}
\usepackage{color}
\newtheorem{theorem}{Theorem}[section]
\newtheorem{lemma}[theorem]{Lemma}

\usepackage{amsmath,amscd}
\usepackage{graphicx}

\newcommand{\inrot}{\mathbin{\rotatebox[origin=c]{90}{$\in$}}}

\newtheorem{example}[theorem]{Example}
\newtheorem{proposition}[theorem]{Proposition}
\newtheorem{problem}[theorem]{Problem}
\newtheorem{corollary}[theorem]{Corollary}
\newtheorem{conjecture}[theorem]{Conjecture}

\theoremstyle{remark}
\newtheorem{remark}[theorem]{Remark}

\numberwithin{equation}{section}

\begin{document}

\title{Eigenconfigurations of Tensors}

\author{Hirotachi Abo}
\address{Department of Mathematics,
University of Idaho,
Moscow, ID 83844, USA}
\email{abo@uidaho.edu}
\thanks{We acknowledge support by the
National Science Foundation (DMS-1419018) and
the US-UK Fulbright Commission.
This project started at the Simons Institute
for the Theory of Computing.
We are grateful to 
Manuel Kauers, Giorgio Ottaviani and
Cynthia Vinzant for their help.}

\author{Anna Seigal}
\address{Department of Mathematics,
University of California,
Berkeley, CA 94720, USA}
\email{seigal@berkeley.edu, bernd@berkeley.edu}

\author{Bernd Sturmfels}

\subjclass{Primary 15A18; Secondary 13P25, 14M12, 15A69}


\keywords{Tensors, eigenvectors, computational algebraic geometry}

\begin{abstract}
Square matrices represent linear self-maps of vector spaces,
and their eigenpoints are the fixed points of the induced map 
on projective space. Likewise, polynomial self-maps of a projective space
are represented by tensors. We study
the configuration of fixed points of a tensor or symmetric tensor.
\end{abstract}

\maketitle

\section{Introduction}

Square matrices $A$ with entries in a field $K$ represent linear maps
of vector spaces, say $K^n \rightarrow K^n$,
and hence linear maps
$\psi: \mathbb{P}^{n-1} \dashrightarrow \mathbb{P}^{n-1}$ of projective spaces over $K$.
If $A$ is nonsingular then $\psi$ is well-defined everywhere,  and the eigenvectors of~$A$
correspond to the fixed points of $\psi$.
The {\em eigenconfiguration} of $A$ consists of $n$ points in $\mathbb{P}^{n-1}$,
provided  $A$ is generic and $K$ is algebraically closed. Conversely, every
spanning configuration of $n$ points in $\mathbb{P}^{n-1}$ arises as the
eigenconfiguration of an $n \times n$-matrix $A$.
However, for special matrices $A$, we obtain multiplicities and
eigenspaces of higher dimensions \cite{AE}. 
Moreover, if $K = \mathbb{R}$
and $A$ is symmetric then its
complex eigenconfiguration consists of real points only.

This paper concerns the extension from linear to non-linear maps.
Their fixed points are the eigenvectors of tensors. The spectral
theory of tensors was pioneered by Lim \cite{Lim} and Qi \cite{Qi}. It is now a
much-studied topic in applied mathematics.

For instance, consider a quadratic map
$\psi : \mathbb{P}^{n-1} \dashrightarrow \mathbb{P}^{n-1}$,
with coordinates
\begin{equation}
\label{eq:quadraticmap}
  \psi_i (x_1,\ldots,x_n)  \,\, = \,\, \sum_{j=1}^n\sum_{k=1}^n a_{ijk} x_j x_k 
\qquad \hbox{for} \,\,i=1,\ldots,n. 
\end{equation}
One organizes the  coefficients of $\psi$ into a tensor $A = (a_{ijk})$
of format $n \times n \times n$.

In what follows, we assume that
$A = (a_{i_1 i_2 \cdots i_d}) $
is a $d$-dimensional tensor of format $n {\times} n {\times} \cdots {\times} n$.
The entries  $a_{i_1 i_2 \ldots i_d}$ lie in an algebraically closed field $K$
of characteristic zero, usually the complex numbers $K = \mathbb{C}$.
Such a tensor $A \in (K^n)^{\otimes d}$ defines polynomial maps $K^n \rightarrow K^n$  
and $\, \mathbb{P}^{n-1} \dashrightarrow \mathbb{P}^{n-1}\,$ 
just as in the formula (\ref{eq:quadraticmap}):
$$ \psi_i (x_1,\ldots,x_n) \,\,\, = \,\, 
\sum_{j_2=1}^n
\sum_{j_3=1}^n
\! \cdots \!
\sum_{j_d=1}^n
a_{i j_2 j_3 \cdots j_d} x_{j_2} x_{j_3} \cdots x_{j_{d}} 
\quad \hbox{for}\,\, i = 1,\ldots,n.  $$
Thus each of the $n$ coordinates of $\psi$ is a homogeneous polynomial $\psi_i$ of
degree $d-1$ in ${\bf x} = (x_1,x_2,\ldots,x_n)$. 
The eigenvectors of $A$ are the solutions of the constraint
\begin{equation}
\label{eq:rank1}
{\rm rank} \begin{pmatrix} x_1 & x_2 & \cdots & x_n \\
\psi_1({\bf x}) & \psi_2 ({\bf x})& \cdots & \psi_n({\bf x}) \end{pmatrix} \,\,\, \leq \,\,\, 1 .
\end{equation}
The {\em eigenconfiguration}  is the variety defined by
the $2 \times 2$-minors of this matrix. 
For a special tensor $A$,  the ideal
defined by (\ref{eq:rank1}) may not be radical, and in that
case we can study its {\em eigenscheme}.
Recent work in \cite{AE} develops this for $d=2$.

We note that every $n$-tuple $(\psi_1,\ldots,\psi_n)$ of  homogeneous
polynomials of degree $d-1$  in $n$ variables can be represented by
   some tensor $A$ as above.
This representation is not unique
unless we require that $A$ is symmetric in the
last $d-1$ indices.
Our maps
$\psi : \mathbb{P}^{n-1} \dashrightarrow \mathbb{P}^{n-1}$
are arbitrary polynomial dynamical 
system on projective space, in the sense of \cite{FS}.
Thus the study of eigenconfigurations of tensors
is equivalent to the study of fixed-point configurations of polynomial maps.

Of most interest to us are {\em symmetric tensors} $A$, i.e.~tensors whose
entries $a_{i_1 i_2 \cdots i_d}$ are invariant
under permuting the $d$ indices. These are in bijection with 
homogeneous polynomials $\phi = \sum a_{i_1 i_2 \cdots i_d} x_{i_1} x_{i_2} \cdots x_{i_d}$,
and we take $\psi_j = \partial \phi/\partial x_j$. 
The eigenvectors of a symmetric tensor correspond
to fixed points of the {\em gradient map}
$\nabla \phi : \mathbb{P}^{n-1} \dashrightarrow \mathbb{P}^{n-1}$, and
our object of study is the variety in $\mathbb{P}^{n-1}$ defined~by
\begin{equation}
\label{eq:rank1sym}
{\rm rank} \begin{pmatrix} x_1 & x_2 & \cdots & x_n \\
\partial \phi/\partial x_1 &
 \partial \phi/\partial x_2 &
 \cdots &
\partial \phi/\partial x_n &
\end{pmatrix} \,\,\, \leq \,\,\, 1 .
\end{equation}

This paper uses the term  {\em eigenpoint}
 instead of eigenvector to stress that we work in $\mathbb{P}^{n-1}$.
In our definition of eigenpoints we include the 
common zeros of $\psi_1,\ldots,\psi_n$. These are the points where
the map $\mathbb{P}^{n-1} \dashrightarrow \mathbb{P}^{n-1}$ is undefined.
For a symmetric tensor $\phi$, they are the singular points
of the hypersurface $\{\phi = 0\}$ in $\mathbb{P}^{n-1}$.
At those points the gradient $\nabla \phi$ vanishes so 
condition (\ref{eq:rank1sym}) holds.
 
\begin{example}
\label{ex:cremona}
 \rm Let $n=d=3$
and $\phi = x y z$.
The corresponding symmetric  $3 {\times} 3 {\times} 3$ tensor $A$
has six nonzero entries $1/6$ and the other $21 $ entries are $0$.
Here $\nabla \phi : \mathbb{P}^2 \dashrightarrow \mathbb{P}^2,
(x:y:z) \rightarrow (yz: xz : xy) $
is the classical {\em Cremona transformation}.
This map has four fixed points, namely
$(1:1:1)$, $(1:1:-1)$, $(1:-1:1)$ and $(-1:1:1)$. Also,
the cubic curve $\{\phi = 0\}$ has the three singular points
$(1:0:0), (0:1:0), (0:0:1)$. In total,
the tensor $A$ has 
seven eigenpoints in $\mathbb{P}^2$.
\end{example}

\smallskip

This paper is organized as follows.
In Section 2 we count the number of
eigenpoints, and we explore eigenconfigurations of
Fermat polynomials, plane arrangements, and binary forms.
Section 3 generalizes the fact that the left eigenvectors and 
right eigenvectors of a square matrix are distinct but compatible.
We explore this compatibility for the
$d$  eigenconfigurations of a $d$-dimensional  tensor with $n=2$.

Section 4 concerns the {\em eigendiscriminant}
 of the polynomial system (\ref{eq:rank1}) and 
 its variant in (\ref{eq:rank1ell}).
This is the irreducible polynomial in the $n^d$ unknowns
$a_{i_1 i_2 \cdots i_d}$ which vanishes when
two eigenpoints come together.
We give a  formula for its degree in terms of $n,d$ and $\ell$.
Section 5 takes first steps towards characterizing eigenconfigurations
among finite subsets of $\mathbb{P}^{n-1}$,
starting with the case $n=d=3$.

In Section 6 we focus on real tensors  and
 their dynamics on real projective space $\mathbb{P}^{n-1}_{\mathbb{R}}$.
 We examine whether all  complex eigenpoints can be real,
and we use line arrangements to give an affirmative answer for $n=3$.
The paper concludes with a brief discussion of attractors for the dynamical systems
$\,\psi: \mathbb{P}^{n-1}_{\mathbb{R}} \dashrightarrow \mathbb{P}^{n-1}_{\mathbb{R}}$.
These are also known as the
{\em robust eigenvectors} of the {\em tensor power method} \cite{AG,Rob}.

\section{The count and first examples}

In this section we assume that the given tensor $A$ is
generic, meaning that it lies in a certain dense open 
subset in the space $(K^{n})^{\otimes d}$ of all 
$n {\times} \cdots {\times} n$-tensors.  This set will be characterized
in Section~4 as the nonvanishing locus of the eigendiscriminant.

\begin{theorem}\label{thm:a}
The number of solutions in $\mathbb{P}^{n-1}$ of the system (\ref{eq:rank1}) equals
\begin{equation}
\label{eq:dustinnum}
 \qquad \frac{ (d-1)^n-1}{d-2}  \,\,\, = \,\,\,\,
\sum_{i=0}^{n-1} (d-1)^i. 
\end{equation}
The same count holds for eigenconfigurations of symmetric tensors, given by (\ref{eq:rank1sym}).
In the matrix case ($d=2$) we use the formula on the right, which evaluates to $n$.
\end{theorem}

This result appeared in the tensor literature in
\cite{CS, OO},
 but it had already been known 
in complex dynamics  due to  Fornaess and Sibony 
  \cite[Corollary 3.2]{FS}.
  We shall present two proofs of Theorem \ref{thm:a},
   cast  in a slightly more general context.
  
For certain applications (e.g.~in spectral hypergraph theory \cite{LQY}), it makes sense to focus
on positive real numbers and to take the $\ell^\mathrm{th}$ root
after each iteration of the dynamical system $\psi$.
This leads to the following generalization of our equations:
\begin{equation}
\label{eq:rank1ell}
{\rm rank} \begin{pmatrix} x_1^\ell & x_2^\ell & \cdots & x_n^\ell \\
\psi_1({\bf x}) & \psi_2 ({\bf x})& \cdots & \psi_n({\bf x}) \end{pmatrix} \,\,\, \leq \,\,\, 1 .
\end{equation}
We refer to the solutions as the {\em $\ell^{\rm th}$ eigenpoints} of the given tensor $A$.
For $\ell=1$, this is the definition in the Introduction.
In the nomenclature devised
by Qi \cite{CQZ, Qi}, one obtains
{\em E-eigenvectors} for  $\ell=1$ and 
 {\em Z-eigenvectors} for $\ell=d-1$.
The subvariety of $\mathbb{P}^{n-1}$ defined by (\ref{eq:rank1ell}) is called the 
{\em $\ell^{\rm th}$ eigenconfiguration} of the tensor $A$.

\begin{theorem}
\label{thm:b}
The $\ell^{\rm th}$ eigenconfiguration 
of a generic tensor $A$ consists of
\begin{equation}
\label{eq:thomformula}
 \frac{(d-1)^n \,-\, \ell^n}{d-1\,\,-\,\,\ell} 
\,\,\, = \,\,\,\,
\sum_{i=0}^{n-1} (d-1)^i \ell^{n-1-i} 
\end{equation}
distinct points in $\mathbb{P}^{n-1}$.  
If $\ell = d - 1$ then the formula on the right is to be used.
\end{theorem}

\begin{proof}
Consider the $2 \times n$-matrix  in (\ref{eq:rank1ell}).
Its rows are filled with homogeneous
polynomials in  $S=K[x_1, \dots, x_n]$
  of degrees $\ell$ and $m$ respectively, where the $\psi_i$ are generic.
Requiring this matrix to have rank $\leq 1$ defines
a subscheme of $\mathbb{P}^{n-1}$. By 
the Thom-Porteous-Giambelli formula \cite[\S 14.4]{Fu1},
this scheme is zero-dimensional, and its 
length is given by the complete homogeneous symmetric
polynomial of degree $n-1$ in the  row degrees, $\ell$ and $m$. 
This is precisely (\ref{eq:thomformula}) if we set $m = d-1$.

Another approach, which also shows that  the scheme is reduced,
is to use vector bundle techniques. 
Consider the $2 \times n$-matrix as a graded $S$-module homomorphism from $S(-\ell) \oplus S(-m)$ to $S^{\oplus n}$. The  quotient module $Q$ of $S^{\oplus n}$ by the submodule generated by the first row 
$(x_1^\ell,\ldots,x_n^\ell)$ is projective. In other words, the sheafification $\widetilde{Q}$ of $Q$ is locally free. The scheme associated with the $2 \times n$-matrix can therefore be thought of as the zero scheme of a generic global section of $\widetilde{Q}(m)$. Since $\widetilde{Q}(m)$ is globally generated, the scheme is reduced \cite[Lemma 2.5]{Ein}. 
\end{proof}

Here is a  brief remark about {\em eigenvalues}.
If ${\bf x} \in K^n $ is an $\ell^{\rm th}$ eigenvector of $A$ then there exists a scalar $\lambda \in K$
such that $\psi_i({\bf x}) = \lambda x_i^\ell $ for all $i$. We call
$({\bf x},\lambda)$ an {\em  eigenpair}. If this holds then 
$(\nu{\bf x}, \nu^{d-1-\ell} \lambda)$ is also an eigenpair for all
$\nu \in K \backslash \{0\}$. 
Such equivalence classes of eigenpairs correspond  to 
the $\ell^{\rm th}$ eigenpoints in $\mathbb{P}^{n-1}$. 
The case $\ell=d-1$ is special because
every eigenpoint has an associated eigenvalue.
If $\ell \not= d- 1$ then eigenpoints make sense
but eigenvalues are less meaningful.

\begin{proof}[Proof of Theorem \ref{thm:a}]
The first statement is the $\ell = 1$ case of Theorem~\ref{thm:b}.
For the second assertion, it suffices to exhibit
one symmetric tensor $\phi$ that has the correct number of  eigenpoints.
We do this for the {\em Fermat polynomial}
\begin{equation}
\label{eq:fermatpoly}
 \phi({\bf x}) \,\, = \,\, x_1^d + x_2^d + \cdots + x_n^d . 
\end{equation}
According to (\ref{eq:rank1}), the eigenconfiguration of $\phi$ is the
variety in $\mathbb{P}^{n-1}$ defined by
\begin{equation}
\label{eq:rank1symb}
{\rm rank} \begin{pmatrix} x_1 & x_2 & \cdots & x_n \\
 x_1^{d-1} & x_2^{d-1} &
  \cdots & x_n^{d-1}
\end{pmatrix} \,\,\, \leq \,\,\, 1 .
\end{equation}
We follow \cite{Rob} in characterizing
all solutions ${\bf x}$ in $\mathbb{P}^{n-1}$ to the binomial equations
$$ x_i x_j ( x_i^{d-2} - x_j^{d-2}) \,\,= \,\, 0 \qquad \hbox{for} \quad 1 \leq i < j \leq n . $$
For any non-empty subset $I \subseteq \{1,2,\ldots,n\}$, there are $ (d-2)^{|I|-1}$ 
solutions ${\bf x}$ with ${\rm supp}({\bf x}) = \{ i \ |\  x_i \not = 0\}$ equal to $I$.
Indeed,  we may assume $x_i = 1$ for the smallest index $i$ in $I$,
and the other values are arbitrary $(d-2)^\mathrm{nd}$ roots of unity. In total,
$$
\sum_I  (d-2)^{|I|-1}  \,= \,
\sum_{i=1}^n \! \binom{n}{i} (d-2)^{i-1} \,= \,
\frac{1}{d{-}2} \sum_{i=1}^n \! \binom{n}{i} (d-2)^{i} 1^{n-i}  \,= \,
\frac{(d{-}2 + 1)^n-1}{d{-}2}.
$$
This equals
(\ref{eq:dustinnum}). 
Here we assume $d\geq 3$.
The  familiar matrix case is $d=2$.
\end{proof}

\begin{example} \rm
Let $d=4$. For each $I$, there are
$2^{|I|-1}$ eigenpoints, with
$x_i = \pm 1$ for $i \in I$ and 
$x_j=0$ for $j \not\in I$. The total number 
of eigenpoints in $\mathbb{P}^{n-1}$ is
$(3^n-1)/2$.
\end{example}

We note that the argument in the proof of Theorem \ref{thm:a}
does not work for $\ell \geq 2$. For instance, if $\ell = d-1$
then every point in $\mathbb{P}^{n-1}$ is
an eigenpoint of  the Fermat polynomial.
At present we do not know an analogue to 
that polynomial for $\ell \geq 2$.

\begin{problem} Given any $\ell, d$ and $n$, exhibit explicit polynomials
$\phi({\bf x})$ of degree $d$ in $n$ variables such that 
(\ref{eq:rank1ell}) has  (\ref{eq:thomformula})
distinct isolated solutions in $\mathbb{P}^{n-1}$.
\end{problem}

We are looking for solutions with 
interesting combinatorial structure.
In Section~6 we shall examine the case when
$\phi({\bf x})$ factors into linear factors,
and we shall see how the geometry of
hyperplane arrangements can be used
to derive an answer.
A first instance was the Cremona map in Example \ref{ex:cremona}.
Here is a second example.

\begin{example} \label{ex:eightyfive}
 \rm
For $n=4$ the count 
of the eigenpoints 
in (\ref{eq:dustinnum}) gives $d^3-2d^2 +2d$.
We now fix $d=5$, so this number equals $85$.
Consider the special symmetric tensor
$$ \phi({\bf x}) \,\, = \,\, x_1 x_2 x_3 x_4 (x_1+x_2+x_3+x_4). $$
The surface defined by $\phi$ consists of five planes in $\mathbb{P}^3$.
These intersect pairwise in ten lines. Each point on such a line
is an eigenpoint because it is singular on the surface. Furthermore,
there are $15$ isolated eigenpoints; these have real coordinates:
\begin{equation}
\label{eq:15eigenpoints}
\begin{small} 
\begin{matrix}
(2:2:-1:-1),\,
(2:-1:2:-1),\,
(2:-1:-1:2),\,
(-1:2:2:-1), \\
(-1:2:-1:2),\,
(-1:-1:2:2), \,\,(1: 1: 1: 1),\,\,
\bigl(\frac{1}{2}(5\pm \sqrt{13}):1:1:1 \bigr) , \,\\
\bigl(1:\frac{1}{2}(5\pm \sqrt{13}):1:1 \bigr),\,
\bigl(1:1:\frac{1}{2}(5\pm \sqrt{13}):1 \bigr),\,
\bigl(1:1:1:\frac{1}{2}(5\pm \sqrt{13} ) \bigr).
\end{matrix} \qquad
\end{small}
\end{equation}
The five planes divide  $\mathbb{P}^3_{\mathbb{R}}$ into
$15$ regions. Each region contains one point in
(\ref{eq:15eigenpoints}).

Now, take a generic quintic $\phi'({\bf x})$ in $\mathbb{R}[x_1,x_2,x_3,x_4]$, and consider
the eigenconfiguration of $\phi({\bf x}) + \epsilon \phi'({\bf x})$. This consists of $85$ points
in $\mathbb{P}^3$.
These are algebraic functions of $\epsilon$. For $\epsilon > 0$ small, we find
$15$ real eigenpoints near (\ref{eq:15eigenpoints}).
The other $70$ eigenpoints arise from the $10$ lines.
How many are real depends on the choice of $\phi'$.
\end{example}

The situation is easier for $n=2$, when the tensor $A$
has format $2 {\times} 2 {\times} \cdots {\times} 2$. It
determines two binary forms $\psi_1$ and $\psi_2$.
The eigenpoints of $A$ are defined by
\begin{equation}
\label{eq:assbinoform}
  y \cdot \psi_1(x,y) - x  \cdot \psi_2(x,y) \, = \, 0 . 
\end{equation}
This is a binary form of degree $d$, so it has $d$ zeros in $\mathbb{P}^1$,
as predicted by (\ref{eq:dustinnum}).
Conversely, every binary form of degree $d$ can be written as
$y \psi_1 - x \psi_2$. This implies:

\begin{remark} \label{rem:everybinary}
Every set of $d$ points in $\mathbb{P}^1$ is the
eigenconfiguration of a  tensor.
\end{remark}

The discussion is more interesting when we restrict ourselves
to symmetric tensors. These correspond to 
binary forms $\phi(x,y)$ and their eigenpoints are defined~by
$$ y \cdot \frac{\partial \phi}{\partial x} \, - \,
x \cdot \frac{\partial \phi}{\partial y} \,\, = \,\, 0 . $$
The matrix case $(d=2)$ shows that Remark \ref{rem:everybinary} cannot hold as stated
for symmetric   tensors. Indeed, if 
$\,A = \begin{pmatrix} a \! &\!  b \\ b \! &\! c \end{pmatrix}\,$
and $\phi = a x^2  + 2 b xy + c y^2$ then
$ \frac{1}{2} (y  \frac{\partial \phi}{\partial x}  - 
x  \frac{\partial \phi}{\partial y}) = -b x^2 + (a-c) xy + b y^2 $.
This confirms the familiar facts that the two eigenpoints $(u_1:v_1)$
and $(u_2:v_2)$ 
are real when $a,b,c \in \mathbb{R}$ and they
satisfy $u_1 u_2 + v_1 v_2 = 0$.
 The following result generalizes the second fact from 
  symmetric matrices to tensors.
 
 \begin{theorem} \label{thm:noteverybinary}
 A set of $d$ points  $(u_i:v_i)$ in $\mathbb{P}^1$
 is the eigenconfiguration of a symmetric tensor if and only if
either $d$ is odd, or $d$ is even and the operator
 $$  \biggl(\frac{\partial^2}{\partial x^2}  + \frac{\partial^2}{\partial y^2}\biggr)^{\! d/2} 
 \,\,\, \hbox{annihilates the corresponding binary form} \quad  \prod_{i=1}^d (v_i x - u_i y).  $$
  \end{theorem}

\begin{proof}
The only-if direction follows from the observation that
the Laplace operator $\frac{\partial^2}{\partial x^2}  + \frac{\partial^2}{\partial y^2}$
commutes with the vector field $y \frac{\partial}{\partial x} - x \frac{\partial}{\partial y}$.
Hence, for any $\phi$ of degree $d$, we obtain zero when
$\,\frac{\partial^2}{\partial x^2}  + \frac{\partial^2}{\partial y^2}\,$ gets applied
$d/2$ times to  $\,y \frac{\partial \phi}{\partial x} \, - \,
x  \frac{\partial \phi}{\partial y} $.

For the if direction, we examine the $(d+1) \times (d+1)$-matrix
that represents the endomorphism $\,\phi \mapsto 
y  \frac{\partial \phi}{\partial x} \, - \,x  \frac{\partial \phi}{\partial y}\,$
on the space of binary forms of degree $d$. This matrix  is
invertible when $d$ is odd, and its kernel is one-dimensional  when $d$ is even.
Hence the map is surjective when $d$ is odd, and it maps onto
a hyperplane when $d$ is even. The only-if part shows that this hyperplane equals
$\bigl(\,\bigl(\frac{\partial^2}{\partial x^2}  + \frac{\partial^2}{\partial y^2}\bigr)^{\! d/2}\,\bigr)^\perp$. 
\end{proof}

After completion of our manuscript we learned that Theorem  \ref{thm:noteverybinary}
was also found independently by Mauro Maccioni, as part of his PhD dissertation at Firenze, Italy.

\begin{example}[$d=4$] \label{ex:notsym} \rm
Four points
$(u_1{:}v_1),(u_2{:}v_2),(u_3{:}v_3),(u_4{:}v_4)$ on the line $\mathbb{P}^1$
arise as the eigenconfiguration of a symmetric
$2 {\times} 2 {\times} 2 {\times} 2$-tensor if and only if
$$ 3 u_1 u_2 u_3 u_4 + u_1 u_2 v_3 v_4 + u_1 u_3 v_2 v_4 + u_1 u_4 v_2 v_3 + \cdots 
+ u_3 u_4 v_1 v_2 + 3 v_1 v_2 v_3 v_4 = 0 .$$
This equation generalizes the orthogonality of the two eigenvectors
of a symmetric $2 {\times} 2$-matrix.
For instance, the columns of
$ U = \begin{small} \begin{pmatrix}
                                1  &  \! 0 &  \!  1 &    1 \\
                                0  &  \! 1 & \!  1  &  \!\!\! -1 \end{pmatrix} \end{small}$
represent the eigenconfiguration of a symmetric
$2 {\times} 2 {\times} 2 {\times} 2$-tensor, but this does not hold for
$ \begin{small} \begin{pmatrix}
1 & 1 \\
0 & 1 \end{pmatrix} U  
\end{small}$.
\end{example}

Example~\ref{ex:notsym} underscores the fact that the constraints on eigenconfigurations 
of symmetric tensors $A$
are not invariant under projective transformations.
They are only invariant under the orthogonal group $O(n)$,
like the Laplace operator in Theorem \ref{thm:noteverybinary}.
By contrast, the constraints on eigenconfigurations
of general (non-symmetric) tensors, such as Theorem
\ref{thm:nosixonconic}, will be 
properties of projective geometry.

We are familiar with this issue
from comparing the eigenconfigurations of
real symmetric matrices with those of
 general square matrices. These are respectively the  $O(n)$-orbit
and the $GL(n)$-orbit of the standard coordinate basis.

\section{Compatibility of eigenconfigurations}

When defining the eigenvectors of a tensor $A$, the symmetry was broken
by fixing the first index and summing over the last $d-1$ indices. There is
nothing special about the first index. For any $k \in \{1,\ldots, d\}$ we 
can regard $A$  as the self-map
$$\psi^{[k]} :\, \mathbb{P}^{n-1} \dashrightarrow \mathbb{P}^{n-1} $$
whose $i^\mathrm{th}$ coordinate is the following homogeneous polynomial of 
degree $d-1$ in~${\bf x}$:
$$ 
\psi^{[k]}_i({\bf x}) \,\, = \,\,
\sum_{j_1=1}^n
\! \cdots \!\!
\sum_{j_{k-1}=1}^n
\sum_{j_{k+1}=1}^n
\!\! \cdots \!
\sum_{j_d=1}^n
a_{j_1 \cdots j_{k-1} i j_{k+1} \cdots j_d} x_{j_1} \cdots x_{j_{k-1}} x_{j_{k+1}} \cdots x_{j_{d}} .
$$
Let ${\rm Eig}^{[k]}(A)$ denote the subvariety of 
$\mathbb{P}^{n-1}$ consisting of the fixed points of $\psi^{[k]}$.
For a generic tensor $A$, this is a finite set of points in $\mathbb{P}^{n-1}$ of cardinality 
$$ D \,\,=  \,\,\frac{ (d-1)^n-1}{d-2} \,\, = \,\, \# ({\rm Eig}^{[k]}(A))
\quad \hbox{for} \,\, d \geq 3 .$$

This raises the following question: Suppose we are given
$d$ configurations, each consisting of $D$ points in $\mathbb{P}^{n-1}$,
and known to be the eigenconfiguration of some tensor.
Under what condition
do they come from the same tensor $A$?

We begin to address this question by considering the case
of matrices $(d=2)$, where $D = n$. Our question is as follows:
given an $n \times n$-matrix $A$, what is the relationship between
the left eigenvectors and the right eigenvectors of $A$?

\begin{proposition}
\label{prop:matrix_compat}
Let $\{{\bf v}_1,{\bf v}_2,\ldots,{\bf v}_n\}$ and
$\{{\bf w}_1,{\bf w}_2,\ldots,{\bf w}_n\}$ be two spanning subsets of  $\mathbb{P}^{n-1}$.
These arise as the left and right eigenconfigurations of some $n \times n$-matrix $A$ if and only if, up to relabeling, the dot products of vectors corresponding to $\,{\bf w}_i$ and ${\bf v}_j\,$ 
are zero whenever $i \not= j$.
\end{proposition}

\begin{proof} Let $V$ be a square matrix whose columns are the eigenvectors of $A$. 
Then the columns of $(V^{-1})^T$ form a basis of eigenvectors for $A^T$.
\end{proof}

The condition in Proposition \ref{prop:matrix_compat}
defines an irreducible variety, denoted ${\rm EC}_{n,2}$ and called
the {\em eigencompatibility variety}
for $n \times n$-matrices. It lives naturally in the space of
pairs of unordered configurations of $n$ points in $\mathbb{P}^{n-1}$.
In symbols,
\begin{equation}
\label{eq:ECn2}
 {\rm EC}_{n,2} \,\subset \, {\rm Sym}_n(\mathbb{P}^{n-1}) \times
{\rm Sym}_n(\mathbb{P}^{n-1}) .
\end{equation}
It has  middle dimension $n(n-1) $,
and it maps birationally onto either factor.
We may identify ${\rm Sym}_n(\mathbb{P}^{n-1})$ 
with the Chow variety of 
products of $n$ linear forms in $n$ variables.
Here, each configuration
$\{{\bf v}_1,{\bf v}_2,\ldots,{\bf v}_n\}$
is represented by $\prod_{i=1}^n ( {\bf v}_i \cdot {\bf x})$.
The coefficients of this homogeneous 
polynomial serve as
coordinates on ${\rm Sym}_n(\mathbb{P}^{n-1})$.
It would be worthwhile to express Proposition~\ref{prop:matrix_compat} 
in these coordinates.

\begin{example}[$n=2$] \label{ex:EC22} \rm
The eigencompatibility variety ${\rm EC}_{2,2}$
for $2 \times 2$-matrices is a surface in
$\,\bigl({\rm Sym}_2(\mathbb{P}^1) \bigr)^2$.
This ambient space equals $(\mathbb{P}^2)^2$,
by representing a pair of unlabeled points on the line
$\mathbb{P}^1$ with the binary quadric that defines it.
To be precise, a point $\bigl((u_0{:}u_1{:}u_2),(v_0{:}v_1{:}v_2) \bigr)$
in $(\mathbb{P}^2)^2$ is identified with the binary forms
$$ f(s,t) \,=\,u_0 s^2 + u_1 st + u_2 t^2
\quad \hbox{and} \quad g(s,t) \,=\,v_0 s^2 + v_1 st + v_2 t^2 . $$
We want the zeros of $f(s,t)$ 
and $g(s,t)$ to be the right and left eigenconfigurations
of the same $2 \times 2$-matrix.
Proposition \ref{prop:matrix_compat} tells us that this is equivalent to 
$$ f(s,t) \,= \, \lambda (as+bt)(cs+dt) 
\quad \hbox{and} \quad g(s,t) \,=\,\,\mu (bs-at)(ds-ct). $$
By eliminating the parameters $a,b,c,d,\lambda$, and $\mu$, we find that
the surface ${\rm EC}_{2,2}$ is essentially the diagonal
in $(\mathbb{P}^2)^2$. It is defined by the determinantal condition
\begin{equation}
\label{eq:mat}
 {\rm rank}
\left(
\begin{array}{rrr}
u_0 & \phantom{-}u_1 & u_2 \\
v_2 & - v_1 & v_0
\end{array}
\right) \,\, \leq \,\, 1.
\end{equation}
Our aim in this section is to generalize this 
implicit representation of ${\rm EC}_{2,2}$.
\end{example}
\smallskip

Let ${\rm EC}_{n,d}$ denote the eigencompatibility variety of
$d$-dimensional tensors of format $n {\times} n {\times} \cdots {\times} n$.
This is defined as follows. Every generic tensor $A$ has
$d$ eigenconfigurations. 
The {\em eigenconfiguration with index $k$} of the tensor $A$
is the fixed locus of the map  $\psi^{[k]}$. Each configuration is a set of unlabeled $D$ points in $\mathbb{P}^{n-1}$, which we regard as a point in ${\rm Sym}_D(\mathbb{P}^{n-1})$.
The $d$-tuples of eigenconfigurations, one for each index $k$, parametrize
\begin{equation}
\label{eq:dtuples}
{\rm EC}_{n,d} \,\subset \,\bigl({\rm Sym}_D(\mathbb{P}^{n-1})\bigr)^d. 
\end{equation}
Thus ${\rm EC}_{n,d}$ is the closure of the locus of 
$d$-tuples of eigenconfigurations of tensors. 

Already the case of binary tensors $(n=2)$ is quite interesting.
We shall summarize what we know about this. 
Let $A$ be a tensor of format $2 \times 2 \times \cdots \times 2 $,
with $d$ factors. Each of its $d$ eigenconfigurations consists of
$D = d$ points on the line $\mathbb{P}^1$. The symmetric power
${\rm Sym}_d(\mathbb{P}^1)$ is identified with the
$\mathbb{P}^d$ of binary forms of degree $d$. The zeros
of such a binary form is an unlabeled configuration of $d$ points in $\mathbb{P}^1$.
Thus, the eigencompatibility variety for binary tensors is
a subvariety 
$$ {\rm EC}_{2,d} \,\,\subset \,\, (\mathbb{P}^d)^d. $$
The case $d=2$ was described in Example \ref{ex:EC22}.
Here are the next few cases.

\begin{example}[$d=3$] \label{ex:EC23}  \rm
Points in $(\mathbb{P}^3)^3$ are triples of binary cubics
$$
\begin{matrix}
f(s,t) & = & u_0 s^3 + u_1 s^2 t+ u_2 s t^2 + u_3 t^3 ,\\
g(s,t) & = & v_0 s^3 + v_1 s^2 t+ v_2 s t^2 + v_3 t^3 ,\\
h(s,t) & = & w_0 s^3 + w_1 s^2 t+ w_2 s t^2 + w_3 t^3 ,
\end{matrix}
$$
where two binary cubics are identified if they differ by a
scalar multiple. The three
eigenconfigurations of a $2 {\times} 2 {\times} 2$-tensor
$A = (a_{ijk})$ are defined by the binary cubics
$$
\begin{matrix}
f(s,t) & \! = &  \lambda \cdot \bigl(a_{211} s^3-(a_{111}-a_{212}-a_{221}) s^2t 
+(a_{222}-a_{112}-a_{121})st^2-a_{122}t^3 \bigr) \\
g(s,t) & \! = &  \mu \cdot \bigl( a_{121} s^3-(a_{111}-a_{122}-a_{221}) s^2 t
+ (a_{222} -a_{112} -a_{211}) st^2 - a_{212}t^3  \bigr) \\
h(s,t) & \! = & \nu \cdot \bigl( a_{112} s^3 -(a_{111}-a_{122}-a_{212}) s^2t
+( a_{222} - a_{121}-a_{211}) st^2 - a_{221} t^3 \bigr )
\end{matrix}
$$
Our task is to eliminate the $11$ parameters $a_{ijk}$ and $\lambda,\mu,\nu$
from these formulas. Geometrically, our variety ${\rm EC}_{2,3}$ is
represented as the image of a rational map
\begin{equation}
\label{eq:visionmap}
 \mathbb{P}^7 \dashrightarrow (\mathbb{P}^3)^3 \,,\,\,A \mapsto (f,g,h) .
 \end{equation}
 This is linear in the coefficients $a_{ijk}$ of $A$ and maps 
 the tensor to a triple of binary forms. 
 To characterize the image of (\ref{eq:visionmap}),
in  Theorem \ref{prop:EC2d} we introduce the matrix
\begin{equation}
\label{eq:3by4uvw}
{\bf E}_3 \,\, = \,\,
\left(
\begin{array}{cccc}
 u_1-u_3 & u_1-u_3 & u_0-u_2 & u_0-u_2 \\
  v_1-v_3 & 0 &   v_0-v_2 &   0 \\
    0 & w_1-w_3 & 0 & w_0-w_2 
    \end{array}\right).
     \end{equation}
Let $I$ be the ideal generated by the $3 \times 3$-minors of ${\bf E}_3$.
 Its zero set has  the eigencompatibility variety 
${\rm EC}_{2,3}$ as an irreducible component. 
There are also three extraneous irreducible components, given by the rows of the matrix:
\[
I_1 = \langle u_0-u_2, u_1-u_3 \rangle, \ I_2 = \langle v_0-v_2, v_1-v_3 \rangle, \ \mbox{and} \ I_3 = \langle w_0-w_2, w_1-w_3 \rangle. 
\] 
The homogeneous prime ideal of ${\rm EC}_{2,3}$ is found to be the ideal quotient 
\begin{equation}
\label{eq:3by4uvwreduced}
(I:I_1 I_2 I_3) \quad = \quad 
\biggl\langle \hbox{$2 \times 2$-minors of} \,\,
\begin{pmatrix}
   u_0-u_2 & v_0-v_2 & w_0-w_2 \\
   u_1-u_3 & v_1-v_3 & w_1-w_3
\end{pmatrix}
\biggr\rangle.
\end{equation}
We conclude that
the eigencompatibility variety ${\rm EC}_{2,3}$ has codimension $2$ in $(\mathbb{P}^3)^3$.
\end{example}

\begin{example}[$d=4$] \label{ex:EC24}  \rm
Points in $(\mathbb{P}^4)^4$ are quadruples of binary quartics
$$
\begin{matrix}
u_0 s^4 + u_1 s^3 t+ u_2 s^2 t^2 + u_3 s t^3 + u_4 t^4 ,\\
 v_0 s^4 + v_1 s^3 t+ v_2 s^2 t^2 + v_3 s t^3  + v_4 t^4 ,\\
 w_0 s^4 + w_1 s^3 t+ w_2 s^2 t^2 + w_3 s t^3 + w_4 t^4 , \\
  x_0 s^4 + x_1 s^3 t+ x_2 s^2 t^2 + x_3 s t^3 + x_4 t^4 . \\
\end{matrix}
$$
One can 
represent the homogeneous ideal of 
the eigencompatibility variety $\mathrm{EC}_{2,4}$ in a similar way to Example~\ref{ex:EC23}. 
Let $I$ be the ideal generated by the $4 \times 4$-minors~of
 \begin{equation}
 \label{eq:4by8uvwx}
 \begin{small}
\begin{bmatrix}
 u_1{-}u_3 & \!\! u_1{-}u_3 & \!\! u_1{-}u_3 & \!\!\! u_2{-}u_0{+}u_4 & \! \!\! u_2{-}u_0{-}u_4 &
  \!\! \!2u_0{-}u_2{+}2u_4 \! & \!\! 3u_0{-}2u_2{+}3u_4 \\
 v_3{-}v_1 &  0  &    0  &  \! \!\! v_0{-}v_2{+}v_4  & 0  &       v_2   &      v_2          \\
      0     &  \! \!\! w_3{-}w_1  & 0  &     0   &    \! \!\! w_0{-}w_2{+}w_4   &  w_2   &      w_2          \\
      0     &  0   &   \! x_3{-}x_1  & 0   &      0   &      x_0+x_4   &   x_2          \\
\end{bmatrix}
\end{small} \!\!\!
\end{equation}
Let $I_{ij}$ be the ideal generated by the $2 \times 2$-minors of the $2 \times 7$-submatrix consisting of the $i^\mathrm{th}$ and $j^\mathrm{th}$ rows in (\ref{eq:4by8uvwx}).
 The homogeneous prime ideal of $\mathrm{EC}_{2,4} \subset (\mathbb{P}^4)^4$ is obtained as the ideal quotient
  $\bigl(I:I_{12}I_{13}I_{14}I_{23}I_{24}I_{34} \bigr)$.
  We obtain ${\rm dim}({\rm EC}_{2,4}) = 12$.
\end{example}

\begin{example}[$d=5$] \label{ex:EC25}  \rm
The eigencompatibility variety ${\rm EC}_{2,5}$ has codimension $4$ in
 $(\mathbb{P}^5)^5$, so ${\rm dim}({\rm EC}_{2,5}) = 21$.
 We represent this variety by the $5 \times 8$-matrix
 \begin{equation}
 \label{eq:5by8uvwxy}
\begin{small}
\begin{bmatrix}
  -u_1+u_3-u_5 & v_1-v_3+v_5 &        0  &    0  &  0 \\
  u_1-u_3+u_5 & 0  &  w_1-w_3+w_5 &    0 & 0 \\
   -u_1+u_3-u_5 &  0  &  0 & x_1-x_3+x_5 & 0 \\
   -u_1+u_3-u_5 &  0  &  0 &  0  & y_1-y_3+y_5 \\
     u_0-u_2+u_4 &  0  & w_0-w_2+w_4 & 0 & 0 \\
        0 & v_0-v_2+v_4 & w_0-w_2+w_4 &  0 & 0 \\
        0  &  0  & w_0-w_2+w_4 & x_0-x_2+x_4 & 0 \\
        0 &  0   & w_0-w_2+w_4 &  0  & y_0-y_2+y_4
        \end{bmatrix} ^{\! T}
         \end{small}
\end{equation}
As before, the variety of maximal minors of this
$5 \times 8$-matrix
has multiple components.
Our variety ${\rm EC}_{2,5}$ is the main component,
obtained by taking the ideal quotient by
determinantal ideals that are given by proper
subsets of the rows.
\end{example}

In what follows we derive a general result
for binary tensors. This will explain the origin
of the matrices  (\ref{eq:3by4uvw}), (\ref{eq:4by8uvwx})
and (\ref{eq:5by8uvwxy}) that were used to represent ${\rm EC}_{2,d}$.

Fix $V = K^n$. 
Tensors $A$ live in the space $V^{\otimes d}$.
For each $k$, the map $A \mapsto \psi^{[k]}$
factors through the linear map
that symmetrizes the factors indexed by $[d] \backslash \{k\}$:
\begin{equation}
\label{eq:firstmap}
 \, V^{\otimes d} \,\longrightarrow \, 
\mathrm{Sym}_{d-1}(V) \otimes V, 
\end{equation}
where  $\{e_1, \dots, e_n\}$ is a basis for $V$. 
Taking the wedge product with $(x_1,x_2,\ldots,x_n)$
defines a further linear map
\begin{equation}
\label{eq:secondmap} 
 \begin{CD}  \mathrm{Sym}_{d-1}(V) \otimes V
  @>>>   \mathrm{Sym}_d(V) \otimes \bigwedge^2 V \qquad \\
  \inrot & & \inrot \\
  \sum_{i=1}^n \psi_i \otimes e_i & \longmapsto & \sum_{1 \leq i < j \leq n} 
(  \psi_i x_j-\psi_j x_i) \otimes (e_i  \wedge e_j).
\end{CD}
 \end{equation}
 Write $ \ell^{[k]}$ for the composition of
  (\ref{eq:secondmap}) after (\ref{eq:firstmap}).
Thus $\ell^{[k]}(A)$ is a vector
of length $\binom{n}{2}$ whose entries
are  polynomials of degree $d$ that define
the  eigenconfiguration with index $k$.
For instance, in Example \ref{ex:EC23},
$f = \ell^{[1]}(A)$,
$g = \ell^{[2]}(A)$,
$h = \ell^{[3]}(A)$.

The kernel of $\ell^{[k]}$ consists of all tensors whose
eigenconfiguration with index $k$ is all of $\mathbb{P}^{n-1}$.
We are interested in the space of tensors where this happens
simultaneously for all indices $k$:
\begin{equation}
\label{eq:alltrivial}
K_{n,d} \,\, = \,\, \bigcap_{k=1}^d {\rm kernel}(\ell^{[k]}).
\end{equation}
The tensors in $K_{n,d}$ can be regarded as being trivial as far
as eigenvectors are concerned. For instance, 
in the classical matrix case $(d=2)$, we have
$$ K_{n,2} \,\,  = \,\, {\rm kernel}(\ell^{[1]}) \,\, = \,\,
{\rm kernel}(\ell^{[2]}) , $$
and this is the  $1$-dimensional space spanned by the identity matrix.

In what follows we restrict our attention to binary tensors ($n=2$).
We regard $\ell^{[k]}$ as a linear map $\mathbb{P}^{2^d-1} \dashrightarrow \mathbb{P}^d$.
The eigencompatibility variety ${\rm EC}_{2,d}$ is the closure of the image of the map
$\mathbb{P}^{2^d-1} \dashrightarrow (\mathbb{P}^d)^d$ given by the tuple 
$(\ell^{[1]},\ldots,\ell^{[d]})$. Let ${\bf u}^{[k]}$ be a column vector of unknowns
representing points in the $k^{\rm th}$ factor $\mathbb{P}^d$.

\begin{theorem} \label{prop:EC2d}
There exists a $d \times e$-matrix ${\bf E}_d$ with $e = \dim (K_{2,d})+d(d+1)-2^d$, 
whose entries in the $k^\mathrm{th}$ row 
are $\mathbb{Z}$-linear forms in ${\bf u}^{[k]}$, such that 
${\rm EC}_{2,d}$ is an irreducible component in the variety 
defined by the $d \times d$-minors of ${\bf E}_d$. Its ideal is
obtained from those $d \times d$-minors by taking  the 
ideal quotient (or saturation) with respect to the
maximal minor ideals of proper subsets of the rows of ${\bf E}_d$.
\end{theorem}

\begin{proof}
We shall derive this using the linear algebra method in \cite[\S 2]{AST}.
We express $\ell^{[k]}$ as a $(d{+}1) \times 2^d$-matrix,
 and we form the $d(d+1) \times (2^d+d)$-matrix
\begin{equation}
\label{eq:aholtmatrix}
\begin{pmatrix}
\, \ell^{[1]} & {\bf u}^{[1]} & 0 & \cdots & 0 \, \\
\, \ell^{[2]} & 0 & {\bf u}^{[2]} & \cdots & 0  \, \\
\, \vdots & \vdots & \vdots & \ddots & \vdots \,\\
\, \ell^{[d]} & 0 & 0 & \cdots & {\bf u}^{[d]}  \,
\end{pmatrix} .
\end{equation}
The left $d(d+1) \times 2^d$-submatrix has entries in $\{-1,0,+1\}$ and its
kernel is  $K_{2,d}$.
The rank of that submatrix is $r = 2^d-{\rm dim}(K_{2,d})$.
Using row operations, we can transform (\ref{eq:aholtmatrix})
into a matrix $
\begin{pmatrix} A & B \\
0 & C \end{pmatrix}$ where $A$ is an $r\times 2^d$ matrix of rank $r$,
and $C$ is an $e \times d$-matrix whose $k^{\rm th}$ column has
linear entries in the coordinates of ${\bf u}^{[k]}$.

The variety ${\rm EC}_{2,d}$ is the set of all points
$({\bf u}^{[1]},{\bf u}^{[2]},\ldots,{\bf u}^{[d]}) $ in $ ( \mathbb{P}^d)^d$ 
such that the kernel of (\ref{eq:aholtmatrix}) contains a vector 
whose last $d$ coordinates are non-zero. Equivalently,
the kernel of $C$ contains a vector whose $d$ coordinates are all non-zero.

Let ${\bf E}_d$ be the transpose of $C$. This is a
$d \times e$-matrix whose $k^{\rm th}$ row has entries that are $\mathbb{Z}$-linear in ${\bf u}^{[k]}$.
By construction, ${\rm EC}_{2,d}$ is the  set of points
$({\bf u}^{[1]},\ldots,{\bf u}^{[d]}) $ in $ ( \mathbb{P}^d)^d$ 
such that ${\bf v} \cdot {\bf E}_d = 0$ for some ${\bf v} \in (K\backslash \{0\})^d$.
This completes the proof. \end{proof}

By our matrix representation, the
 codimension of ${\rm EC}_{2,d}$  is at most $e-d+1$, so
\begin{equation}
\label{eq:dimlowerbound}
 {\rm dim}({\rm EC}_{2,d}) \,\,\geq \,\, d^2-(e-d+1). 
 \end{equation}
Examples~\ref{ex:EC22}, \ref{ex:EC23}, and~\ref{ex:EC24} suggest that (\ref{eq:dimlowerbound}) is an equality.
\begin{conjecture}
The dimension of $\,{\rm EC}_{2,d}$ equals $\, d^2-(e-d+1)$. 
\end{conjecture}
We do not know the dimension
of $K_{2,d} = \bigcap_k {\rm kernel}(\ell^{[k]})$.
In our examples, we saw that ${\rm dim}(K_{2,d}) = 1,0,3,10$ for $d=2,3,4,5$ respectively.
It would be desirable to better understand the common kernel $K_{n,d}$
for arbitrary $n$ and $d$:

\begin{problem}
Find the dimension of the space $K_{n,d}$ in  (\ref{eq:alltrivial}).
\end{problem}

Another problem is to understand the
diagonal of ${\rm EC}_{n,d}$ 
in the embedding (\ref{eq:dtuples}).
This diagonal parametrizes 
{\em simultaneous eigenconfigurations},
arising from special tensors $A$ whose $d$ maps
$\psi^{[1]},\ldots,\psi^{[d]}$
all have the same fixed point locus in $\mathbb{P}^{n-1}$.
Symmetric tensors $A$ have this property, and the
issue is to characterize all others.

\begin{example}[$n=2$] \rm
The diagonal of ${\rm EC}_{2,d}$
is computed by setting
${\bf u}^{[1]} = {\bf u}^{[2]} = \cdots = {\bf u}^{[d]}$
in the prime ideal  described in Theorem \ref{prop:EC2d}.
If $d$ is odd, then there is no constraint, by Theorem \ref{thm:noteverybinary}.
However, for $d$ even, the diagonal of ${\rm EC}_{2,d}$ is interesting.
For instance, for $d=2$, equating the rows in (\ref{eq:mat}) gives
two components
$$
\biggl\langle \hbox{$2 \times 2$-minors of} \,\,
\begin{pmatrix}
a_0 & \phantom{-}a_1 & a_2 \\
a_2 & - a_1 & a_0
\end{pmatrix}
\biggr\rangle \,\, = \,\,\,
\bigl\langle a_0 + a_2 \bigr\rangle \,\cap \,\bigl\langle a_0 - a_2, a_1 \bigr\rangle. $$
The first component is the known case of symmetric $2 \times 2$-matrices.
The second component is a point in ${\rm Sym}_2(\mathbb{P}^1)$,
namely the binary form $s^2 + t^2 = (s-it)(s+it) $.
This is the simultaneous eigenconfiguration of any matrix 
$A = \begin{small} \begin{pmatrix} \phantom{-} a & \! b\, \\ -b & \! a \,\end{pmatrix}\end{small}$
with $b \not = 0$.
\end{example}

\section{The eigendiscriminant}

The $d$-dimensional tensors of format $n {\times} n {\times} \cdots {\times} n$
represent points in a projective space $\mathbb{P}^N$ where $N = n^d-1$.
For a generic tensor $A \in \mathbb{P}^N$, the $\ell^{\rm th}$ eigenconfiguration,
in the sense of (\ref{eq:rank1ell}),
consists of a finite set of reduced points in $\mathbb{P}^{n-1}$.
We know from  Theorem \ref{thm:b}
that the number of these points   equals
$$ \rho(n,d,\ell) \,=\, \sum_{i=0}^{n-1} (d-1)^i \ell^{n-1-i}. $$

In this section we study the set $\Delta_{n,d,\ell}$ of all tensors $A$ for which 
the eigenconfiguration consists of fewer than $\rho(n,d,\ell)$ points
or is not zero-dimensional.
This set is a subvariety of $\mathbb{P}^N$, called the
{\em $\ell^{\rm th}$ eigendiscriminant}. We also abbreviate
\begin{equation}
\label{eq:genus}
 \gamma(n,d,\ell) \,\, = \, \,
\sum_{j=2}^{n-1} (-1)^{n-1+j}(j{-}1) 
\left[\sum_{k=0}^j(-1)^k {n \choose j{-}k}{j(d{-}1)-k\ell-1 \choose n-1}\right].
\end{equation}
The following is our main result in this section:

\begin{theorem} \label{thm:discdeg}
The $\ell^{\rm th}$ eigendiscriminant  is an irreducible hypersurface with
\begin{equation}
\label{eq:discdeg}
{\rm degree}(\Delta_{n,d,\ell}) \quad = \quad
2\gamma(n,d,\ell)+2\rho(n,d,\ell)-2.
\end{equation}
\end{theorem}

We identify $\Delta_{n,d,\ell}$ with the unique (up to sign) irreducible polynomial 
with integer coefficients in the $n^d$ unknowns $a_{i_1 i_2 \cdots i_d}$
that vanishes on this hypersurface. 
From now on we use the term eigendisciminant to refer to the 
polynomial $\Delta_{n,d,\ell}$.

The case of most interest is $\ell= 1$,
which pertains to the eigenconfiguration of 
a tensor in the usual sense of (\ref{eq:rank1}).
For that case, we write $\Delta_{n,d} = \Delta_{n,d,1}$
for the eigendiscriminant, and the formula for its degree 
can be simplified as follows:

\begin{corollary} \label{cor:kauers}
The eigendiscriminant is a homogeneous polynomial of degree 
$$ {\rm degree}(\Delta_{n,d}) \,\, = \,\, n(n-1)(d-1)^{n-1}. $$
\end{corollary}

The following proof is due to Manuel Kauers.
We are grateful for his help.

\begin{proof}
We set  $\ell = 1$ in the expression (\ref{eq:discdeg}).
Our claim is equivalent to
\begin{equation}
\label{eq:desired}
    \gamma(n,d,1) \,\,\,= \,\,\,\,\binom n2(d-1)^{n-1} \,-\, \frac{(d-1)^n-1}{d-2}\,+\,1.
    \end{equation}
We abbreviate the innermost summand in (\ref{eq:genus})  as
  \[
    s_{n,d,j}(k) \,\,= \,\, (-1)^k \binom n{j-k}\binom{j(d-1)-k-1}{n-1}. 
  \]
  Using Gosper's algorithm
  \cite[Chapter 5]{PWZ}, we find the multiple
   $$ S_{n,d,j}(k) \,\,:= \,\, \frac{(j(d-1)-k)(j-k-n)}{nj(d-2)}s_{n,d,j}(k) .$$
   It can now be checked by hand that this satisfies
  $$ S_{n,d,j}(k+1) - S_{n,d,j}(k) \,\,= \,\,s_{n,d,j}(k). $$
  Summing over the range $k=0,\dots,j-1$ and simplifying expressions lead to
  \begin{equation}
  \label{eq:innersum}    \quad \sum_{k=0}^j s_{n,d,j}(k)
    \,= \,S_{n,d,j}(j+1)  - S_{n,d,j}(0) 
\,\,    = \,\, \frac{d-1}{d-2}\binom {n-1}j\binom{j(d-1)-1}{n-1}.
\end{equation}
This is  valid for all $j\geq1$. 

  Next we introduce the expression
  \begin{equation}
  \label{eq:defineA}
    A(n,d) \,\,= \,\,\sum_{j=0}^n (-1)^j \binom nj \binom{d\,j-1} n (j-1).
 \end{equation}
Consider $\binom{d\,j-1}n(j-1)$ as a polynomial in $j$ of degree $n+1$. In the binomial basis,
  \begin{alignat*}1
    \binom{d\,j-1}n(j-1) 
    &\,=\, (n+1)d^n\binom j{n+1}
    + \biggl((d-1)d^{n-1}\binom {n+1}2 - d^n\biggr)\binom jn\\
    &\quad{}+ \text{lower degree terms},
  \end{alignat*}
  Recall from \cite[page 190]{GKP} that
  \[    \sum_{j=0}^n (-1)^j \binom nj \binom jk \,\,=\,\, \left\{\begin{array}{ll}
        (-1)^n &\text{ if $k=n$,}\\
        0 &\text{ otherwise.}
      \end{array}\right.  \]
  This implies 
  $$ A(n,d) \,\,= \,\,(-1)^n \biggl((d-1)d^{n-1}\binom {n+1}2 - d^n\biggr).   $$
  
  Combining this identity with (\ref{eq:genus}),  (\ref{eq:innersum}) and (\ref{eq:defineA}), we now derive
  \begin{alignat*}1
   \gamma(n,d,1) & =  (-1)^{n-1}\frac{d-1}{d-2}\biggl(A(n-1,d-1) + \binom{-1}{n-1}\biggr) \\
    & = (-1)^{n-1}\frac{d{-}1}{d{-}2}\biggl( \! (-1)^{n-1} \biggl(\!(d{-}2)(d{-}1)^{n-2}\binom {n}2 - (d{-}1)^{n-1}
    \! \biggr) + (-1)^{n-1} \!\biggr)\\
    &=\frac{d-1}{d-2}\biggl((d-2)(d-1)^{n-2}\binom {n}2 - (d-1)^{n-1}\biggr) + \frac{d-1}{d-2}\\
    &=\binom n2 (d-1)^n - \frac{(d-1)^n - (d-1)}{d-2}.
  \end{alignat*}
  This equals the desired expression for $\gamma(n,d,1)$
  on the right hand side of (\ref{eq:desired}).
  \end{proof}

The proof of Theorem \ref{thm:discdeg}
involves some algebraic geometry
and will be presented later in this section.
We first discuss a few examples to illustrate $\Delta_{n,d}$.

\begin{example}[$d=2$] \rm
The eigendiscriminant of an $n \times n$-matrix $A = (a_{ij})$
is the discriminant of its characteristic polynomial. In symbols,
$$ \Delta_{n,2} \,\, = \,\, {\rm discr}_\lambda \bigl(\,{\rm det}(\,A - \lambda \cdot {\rm Id}_n)\,\bigr) .$$
This is a homogeneous polynomial of degree $n(n-1)$ in the
matrix entries $a_{ij}$.
For instance, for a $3 \times 3$-matrix, the eigendiscriminant is a 
polynomial with $144$ terms:
$$ \Delta_{3,2} \,= \,
a_{11}^4 a_{22}^2-2 a_{11}^4 a_{22}a_{33}+
4 a_{11}^4 a_{23}a_{32} + a_{11}^4 a_{33}^2 - 2 a_{11}^3 a_{12} a_{21} a_{22}
+ \cdots + a_{23}^2 a_{32}^2 a_{33}^2.
$$
This polynomial vanishes whenever two of the eigenvalues of $A$ coincide.

There is a beautiful theory behind $\Delta_{n,2}$
in the case when $A$ is real symmetric, so
the eigenconfiguration is defined over $\mathbb{R}$.
The resulting {\em symmetric eigendiscriminant}
is a nonnegative polynomial of degree $n(n-1)$
in the  $\binom{n+1}{2}$ matrix entries.
Its real variety has codimension $2$ 
and degree $\binom{n+1}{3}$,  and its determinantal representation
governs expressions of
$\Delta_{n,2}$ as a sum of squares of polynomials
of degree $\binom{n}{2}$.
For further reading on this topic see \cite[Section 7.5]{Stu}
and the references given there.
\end{example}

\begin{example}[$n=2$] \rm
The eigendiscriminant of a $d$-dimensional tensor of format
$\,2 {\times} 2 {\times} \cdots {\rm \times} 2 \,$ is the discriminant
of the associated binary form in (\ref{eq:assbinoform}), i.e.
$$
\Delta_{2,d} \,\, = \,\,
{\rm disc}_{(x,y)} \bigl(  \,y \cdot \psi_1(x,y) - x  \cdot \psi_2(x,y) \,\bigr).
$$
This is a  homogeneous polynomial of degree
$2d-2$ in the $2^d$ tensor entries $a_{i_1 i_2 \cdots i_d}$.
\end{example}

\begin{example}[$n=d=3$] \rm
The eigendiscriminant $\Delta_{3,3}$ of
a $3 {\times} 3 {\times} 3$-tensor $A = (a_{ijk})$
is a homogeneous polynomial of degree $24$ in the
$27$ entries $a_{ijk}$.
If we specialize $A$ to a symmetric tensor, corresponding to a
ternary cubic
$$ \phi(x,y,z) \,\, = \,\,\, c_{300} x^3 + c_{210} x^2 y + c_{201} x^2 z + \cdots + c_{003} z^3, $$
then $\Delta_{3,3}$ remains irreducible. The resulting irreducible polynomial of degree $24$ 
in the ten coefficients $c_{ijk}$
is the  {\em eigendiscriminant of a tenary cubic}.
At present we do not know an explicit formula for $\Delta_{3,3}$, but it is 
fun to explore specializations of the eigendiscriminant. For instance, 
if $\, \phi = (2x+y)(2x+z)(2y+z)+u \cdot xyz\,$ then
\begin{tiny} 
$$ \begin{matrix} \Delta_{3,3} \,\, = \,\, 
16 u^{24}+2304 u^{23}+152784 u^{22}+6097536 u^{21}+159761808 u^{20}+
      2779161840 u^{19}+29727588168 u^{18} \\ +124641852624 u^{17}  -1234078589016 u^{16}-
      18314627517360 u^{15}-8929524942432 u^{14}+1200933047925648 u^{13} \\ +
      3722203539791685 u^{12} -63418425922462464u^{11} -257381788882972176u^{10}  +
      2676970903961440800 u^9 \\ +7927655114836286496u^8  -89013482239908955392u^7-
      13934355026171012352 u^6 \\ +1729250356371556792320u^5  -5159222324901192930048 u^4
       -11838757458480721920 u^3 \\ +28255456641734116982784 u^2-
      56809371779894977339392 u+37304830510913780269056 ,
      \end{matrix}
  $$    \end{tiny} 
  and if $\,\phi = u \cdot x^3+v \cdot y^3+w \cdot z^3 + x yz\,$ then $\Delta_{3,3}$ is the square of polynomial
  \begin{tiny}  
$$ \begin{matrix} 
 531441u^4v^4w^4-708588u^5v^3w^3-708588u^3v^5w^3-708588u^3v^3w^5
     +1062882u^4v^4w^2+1062882u^4v^2w^4 \\ +1062882u^2v^4w^4-1810836u^3v^3w^3
     -177147u^4v^4+39366u^4v^2w^2+39366u^2v^4w^2-177147u^4w^4 \\ +
     39366u^2v^2w^4 -177147v^4w^4+314928u^3v^3w+314928u^3vw^3+314928uv^3w^3 
     {-}244944u^2v^2w^2{-}46656u^3vw \\ -46656uv^3w {-} 46656uvw^3{+}23328u^2v^2
     {+}23328u^2w^2{+}23328v^2w^2{+}6912uvw{-}2304u^2-2304v^2-2304w^2+256.
\end{matrix}
$$
\end{tiny}
\end{example}

We now embark towards the proof of Theorem \ref{thm:discdeg}.
Let $A$ and $B$ be generic tensors of the same format, and write
$(\psi_1,\ldots,\psi_n)$ and
$(\omega_1,\ldots,\omega_n)$ for the 
vectors of degree $d-1$ polynomials that represent
the corresponding maps $\mathbb{P}^{n-1} \dashrightarrow \mathbb{P}^{n-1}$.
Let $\mathcal{C}$ denote the subvariety of $\mathbb{P}^{n-1}$ defined by
the determinantal constraints
\begin{equation}
\label{eq:rank2}
{\rm rank} \begin{pmatrix} 
\psi_1({\bf x}) & \psi_2 ({\bf x})& \cdots & \psi_n({\bf x}) \\
\omega_1({\bf x}) & \omega_2 ({\bf x})& \cdots & \omega_n({\bf x}) \\
x_1^\ell & x_2^\ell & \cdots & x_n^\ell \\
\end{pmatrix} \,\,\, \leq \,\,\, 2 .
\end{equation}
Since the $\psi_i$ and $\omega_j$ are generic,
this defines a variety of codimension $n-2$.
We find that $\mathcal{C}$ is a curve that is smooth and irreducible,
by an argument similar to that in the proof of Theorem \ref{thm:b}.
The following lemma is the key to~Theorem \ref{thm:discdeg}.

\begin{lemma} \label{lem:genusformula}
The expression in (\ref{eq:genus})
is the genus of the curve $\mathcal{C}$. In symbols,
$$ {\rm genus}(\mathcal{C}) \, = \,\gamma(n,d,\ell). $$
\end{lemma}

Using this lemma, we now derive the degree of the eigendiscriminant.

\begin{proof}[Proof of Theorem \ref{thm:discdeg}]
We define a map $\mu:\mathcal{C} \rightarrow \mathbb{P}^1$ as follows.
For any point ${\bf x}$ on the curve $\mathcal{C}$, the matrix
in (\ref{eq:rank2}) has rank $2$, so, up to scaling, there exists a unique 
row vector $(a,b,c) \in K^3$ that spans the left kernel of that $3
\times n$-matrix. We define the image of ${\bf x} \in \mathcal{C}$ to be the point
$\mu({\bf x}) = (a : b)$ on the projective line~$\mathbb{P}^1$. 
This condition means that ${\bf x}$ is an eigenpoint of the tensor $aA + bB$.
Conversely, for any $(a:b) \in \mathbb{P}^1$, the fiber
$\mu^{-1}(a:b)$ consists precisely of the eigenpoints of $aA + bB$.
Hence, since $A$ and $B$ are generic, the generic fiber is finite and reduced of
cardinality  $\rho(n,d,\ell)$. In other words, $\mu:\mathcal{C} \rightarrow \mathbb{P}^1$
is a map of degree $\rho(n,d,\ell)$.

We restrict the eigendiscriminant to our $\mathbb{P}^1$ of tensors.
The resulting binary form $\Delta_{n,d,\ell}(aA + bB)$ is squarefree,
and its degree is the left hand side in (\ref{eq:discdeg}).
The points $(a:b) \in \mathbb{P}^1$ where 
 $\Delta_{n,d,\ell}(aA + bB) = 0$ are the {\em branch points} of the map $\mu$.
 The corresponding multiplicity-two eigenpoints ${\bf x}$ form the
 {\em ramification divisor} on $\mathcal{C}$. The number of branch points of $\mu$
 is the degree of the eigendiscriminants $\Delta_{n,d,\ell}$.
 
  The Riemann-Hurwitz Formula
\cite [Exercise 8.36]{fulton} states that the number of branch points 
of the map $\mu: \mathcal{C} \rightarrow \mathbb{P}^1$ is
$2 \cdot {\rm genus}(\mathcal{C}) + 2 \cdot  {\rm degree}(\mu)  - 2 $.
By the first paragraph, and by Lemma \ref{lem:genusformula},
this expression is the right hand side of~(\ref{eq:discdeg}).
\end{proof}

Our proof of Lemma~\ref{lem:genusformula} is fairly complicated,
and we decided not to include it here. It is  based on resolutions of 
vector bundles, like those seen in the proof of Theorem \ref{thm:b}.
We plan to develop this further and publish it in a later paper on
discriminants arising from maximal minors of matrices 
with more than two rows.

What we shall do instead is to prove an alternative
combinatorial formula for the genus of $\mathcal{C}$
that is equivalent to (\ref{eq:genus}).
This does not prove Lemma~\ref{lem:genusformula}
because we presently do not know a direct argument
to show that they are equal.
Nevertheless, the following discussion is an
illustration of useful commutative algebra techniques.

\begin{proof}[Instead of Lemma \ref{lem:genusformula}] 
The Hilbert polynomial $H_\mathcal{C}(t)$ of the curve $\mathcal{C}$ equals
$$ H_{\mathcal{C}}(t) \,\,\, = \,\,\, {\rm degree}(\mathcal{C}) \cdot t \,+\,  (1-{\rm genus}(\mathcal{C})). $$
Recall that $\mathcal{C}$ is a linear section of the
variety defined by the maximal minors of a $3 \times n$-matrix
whose rows are homogeneous of degrees $d-1,d-1$ and $\ell$.
That variety is Cohen-Macaulay. 
We shall compute the Hilbert polynomial of the coordinate ring of $\mathcal{C}$ from its
graded minimal free resolution over $S = K[x_1,\ldots, x_n]$.

Consider the $S$-linear map from $\,F=S^{\oplus n}\,$ to
  $\,G = S(d-1)^{\oplus 2} \oplus S(\ell)\,$ given by
\begin{equation}
\label{eq:homomorphism}
\alpha \,= \,
\begin{pmatrix} 
\psi_1({\bf x}) & \psi_2 ({\bf x})& \cdots & \psi_n({\bf x}) \\
\omega_1({\bf x}) & \omega_2 ({\bf x})& \cdots & \omega_n({\bf x}) \\
x_1^\ell & x_2^\ell & \cdots & x_n^\ell \\
\end{pmatrix}.  
\end{equation}
By \cite[Section A2H]{Eis}, the
corresponding {\em  Eagon-Northcott complex}
 $\,\mathrm{EN}(\alpha)$ equals
\[
0 \rightarrow \mathrm{Sym}_{n-3} (G^\vee) \otimes \bigwedge^n \! F \rightarrow \mathrm{Sym}_{n-4} (G^\vee) \otimes \bigwedge^{n-1} \!\! F \rightarrow \cdots \rightarrow G^\vee \otimes \bigwedge^4 \!
F \rightarrow \bigwedge^3 \! F \rightarrow \bigwedge^3 \! G, 
\]
where $\mathrm{Sym}_i (G^\vee)$ is the $i^\mathrm{th}$ symmetric power of $G^\vee$ and $\bigwedge^i F$ is the $i^\mathrm{th}$ exterior power of $F$. 
We compute the Hilbert polynomial $H_M(t)$ of each module $M$ in $\mathrm{EN}(\alpha)$.

Since $\mathcal{C}$ has codimension 
$n-2 = \mathrm{rank} \, F - \mathrm{rank} \, G$+1, 
the complex $\mathrm{EN}(\alpha) \otimes S(-2d-\ell+2)$ is a free resolution of the coordinate ring of $\mathcal{C}$. 
In particular, 
\begin{equation}
\label{eq:HC} 
H_\mathcal{C}(t) \,\,= \,\,
H_{\bigwedge^3 G}(t-2d-\ell+2) + \sum_{j=3}^n (-1)^j H_{E_j} (t-2d-\ell+2),
\end{equation}
where $E_j =  \mathrm{Sym}_{j-3} (G^\vee) \otimes \bigwedge^j F $.
Since $G^\vee = S(-d+1)^{\oplus 2} \oplus S(-\ell)$ and $F = S^{\oplus n}$, 
\[
\mathrm{Sym}_{j-3} (G^\vee)  =   \bigoplus_{k=0}^{j-3} S\left((j-k-3)(-d+1)-k\ell\right)^{\oplus j-k+2} \ \mbox{and} \ \bigwedge^j F  =  S^{\oplus {n \choose j}}.
\]
Their tensor product is the $j^{\rm th}$ term in $\mathrm{EN}(\alpha)$. As a graded $S$-module, it equals
$$ E_j  \,\,= \,\,\,
 \bigoplus_{k=0}^{j-3} S\left((j-k-3)(-d+1)-k\ell\right)^{\oplus (j-k-2){n \choose j}}.  
$$
The shifted Hilbert series of this module is the summand on the right of
(\ref{eq:HC}):
\[
H_{E_j}(t-2d-\ell+2) \,=\, \sum_{k=0}^{j-3} (j-k-2){n \choose j}
{t+(n{-}1)+(j{-}k{-}1)(1{-}d)-(k{+}1)\ell \choose n-1}.
\]
We conclude that the Hilbert polynomial 
$H_\mathcal{C}(t)$ of the curve $\mathcal{C}$ equals
\[
 {t{+}n{-}1 \choose n-1} + \sum_{j=3}^n (-1)^j \sum_{k=0}^{j-3} (j-k-2){n \choose j}
{t+(n{-}1)+(j{-}k{-}1)(1{-}d)-(k{+}1)\ell \choose n-1}. 
\]
The genus of $\mathcal{C}$ is obtained by substituting $t=0$ and subtracting
the result from $1$:
$$
 \gamma(n,d,\ell) \, = \,
\sum_{j=3}^n (-1)^{j-1} \sum_{k=0}^{j-3} (j{-}k{-}2){n \choose j}
{(n{-}1)+(j{-}k{-}1)(1{-}d)-(k{+}1)\ell\choose n-1}. 
$$
This formula is equivalent to  (\ref{eq:genus}).
\end{proof}

\begin{example}[$n=4$] \rm
The last formula seen above specializes to
$$ \gamma(4,d,\ell) \,\, = \,\,
 4 \binom{3+2(1{-}d)-\ell}{3} - 
\biggl[2 \binom{3+3(1{-}d)-\ell}{3} + \binom{3+2(1{-}d)-2\ell}{3} \biggr],
$$
while the genus formula in (\ref{eq:genus}) states
$$
\begin{matrix} \gamma(4,d,\ell) \,\,= & 
- \biggl[ 6\binom{2(d-1)-1}{3}-4 \binom{2(d-1)-\ell-1}{3}+ \binom{2(d-1)-2\ell-1}{3} \biggr] \qquad
\qquad \qquad \\ &
+ 2  \biggl[ 4\binom{3(d-1)-1}{3}-6 \binom{3(d-1)-\ell-1}{3}
+ 4 \binom{3(d-1)-2\ell-1}{3}
- \binom{3(d-1)-3\ell-1}{3} \biggr].
\end{matrix}
$$
Both of these evaluate to the cubic polynomial
$$ \gamma(4,d,\ell) \,= \,
5 d^3+5 d^2 \ell+3 d \ell^2 + \ell^3-21d^2-14d \ell-5 \ell^2+27d+9\ell-10.
$$
Therefore, by Theorem \ref{thm:discdeg}, the degree of the $\ell^{\rm th}$ 
eigendiscriminant for $n=4$ equals
$$  {\rm degree}(\Delta_{4,d,\ell}) \,\,=\,\,
3 d^3+3 d^2 \ell+2 d \ell^2+ \ell^3-12d^2-8d\ell-3\ell^2+15d+5\ell-6.
$$
For $\ell=1$, this factorizes as promised in Corollary \ref{cor:kauers}:
 $ \, {\rm degree}(\Delta_{4,d,1}) = 12 (d-1)^3$.
\end{example}

\section{Seven points in the plane}

Our study had been motivated by the
desire to find a geometric characterization of
eigenconfigurations among all
finite subsets of $\mathbb{P}^{n-1}$.
The solution for $n=2$ was presented in
Theorem \ref{thm:noteverybinary}. However, the
relevant geometry is more difficult in higher dimensions.
 In this section we take some steps towards a characterization  for $n=3$.
The eigenconfiguration of a general tensor $A$ in  $(K^3)^{\otimes d}$
 consists of $d^2-d+1$ points in $\mathbb{P}^2$. So, our question 
 can be phrased like this: given a configuration $Z \in (\mathbb{P}^2)^{d^2-d+1}$,
  decide whether it is an eigenconfiguration. If yes, 
 construct a corresponding tensor $A \in  (K^3)^{\otimes d}$,
 and decide whether $A$ can be chosen to be symmetric. 
 
 The first interesting case is $d=n=3$. Here the following result holds.
 
 \begin{theorem} \label{thm:nosixonconic}
 A configuration of seven points in $\mathbb{P}^2$ is the eigenconfiguration of a 
 $3 \times 3 \times 3$-tensor if and only if no six of the 
 seven
 points lie on a conic.
 \end{theorem}

The only-if part of this theorem appears also in \cite[Proposition~2.1]{OS}, 
where Ottaviani and Sernesi studied the degree $54$ hypersurface 
of all {\it L\"uroth quartics} in~$\mathbb{P}^2$.
We note that part (i) in  \cite[Proposition~2.1]{OS} is not quite correct.
A counterexample is the configuration $Z$ consisting of
four points on a line and three other general points. It is precisely this
gap that makes our proof of Theorem~\ref{thm:nosixonconic} a bit  lengthy.

This proof will be presented later in this section.
Example~\ref{ex:cremona} shows that some triples among
the seven eigenpoints  in $\mathbb{P}^2$ can be collinear.
Another interesting point is that being an
eigenconfiguration is not a closed condition. For  a general $d$ it makes sense
to pass to the Zariski closure. We define ${\rm Eig}_d$ to be the closure in $(\mathbb{P}^2)^{d^2-d+1}$
of the set of all eigenconfigurations. Readers from algebraic geometry may prefer
unlabeled configurations, and they  would take the closure 
in the Chow variety ${\rm Sym}_{d^2-d+1}(\mathbb{P}^2)$
or in the Hilbert scheme ${\rm Hilb}_{d^2-d+1}(\mathbb{P}^2)$.
 For simplicity of exposition, we work
in the space of labeled point configurations.
We also consider the {\em variety of symmetric eigenconfigurations},  denoted
${\rm Eig}_{d,{\rm sym}}$. This is the Zariski closure  in $(\mathbb{P}^2)^{d^2-d+1}$ of
the set of eigenconfigurations of ternary forms $\phi$ of degree~$d$.
Towards the end of this section we examine the dimensions of ${\rm Eig}_d$ 
and  ${\rm Eig}_{d,{\rm sym}}$.

\smallskip

We begin by approaching our problem with a pinch of
commutative algebra. 
Let $Z \in (\mathbb{P}^2)^{d^2-d+1}$ and write
$I_Z$ for the ideal of all polynomials in $S = K[x,y,z]$
that vanish 
at all points in the configuration $Z$. 
This homogeneous radical ideal is Cohen-Macaulay because it has a free resolution of length $1$ (see, for example, \cite[Proposition 3.1]{Eis}). By the Hilbert-Burch Theorem,
the minimal free resolution of $I_Z$  has the form
$$ 0 \,\rightarrow \,S^{\oplus (m-1)} \,
\stackrel{\varPhi}{\longrightarrow}
\, S^{\oplus m} \,\rightarrow\, I_Z \,\rightarrow \, 0. $$
The $m \times (m{-}1)$-matrix $\varPhi$ is the {\em Hilbert-Burch matrix} of $Z$.
The minimal free resolution of $I_Z$ is unique up to 
change of bases in the graded $S$-modules.
 In that sense, we write $\varPhi_Z := \varPhi$.
The ideal $I_Z$ is generated by the maximal minors of~$\varPhi_Z$. The following proposition is due to Ottaviani and Sernesi (see~\cite[Proposition~2.1]{OS}). 

\begin{proposition} \label{prop:asrequired}
Let $Z$ be a configuration in $ (\mathbb{P}^2)^{d^2-d+1}$. Then $Z$
is the eigenconfiguration of a tensor if and only if its Hilbert-Burch matrix has the form
\begin{equation}
\label{eq:asrequired}
\varPhi_Z \,\,= \,
\left(
\begin{array}{cc}
L_1 & F_1 \\
L_2 & F_2 \\
L_3 & F_3 \\
\end{array} \right),
\end{equation}
where 
$L_1,L_2,L_3$ are linear forms that are linearly independent over $K$.
\end{proposition}

This statement makes sense because the condition on $\varPhi_Z$
is invariant under row operations over $K$. 
 The ternary forms $F_1,F_2,F_3$ must all have the same degree, and 
the hypothesis on $Z$ ensures that this common degree is $d-1$.

\begin{proof}
We start with the only-if direction. Suppose that $Z$ is an eigenconfiguration.
Then there exist ternary forms $\psi_1, \psi_2, \psi_3$ of degree $d-1$ such that
$Z$ is defined set-theoretically by the $2 \times 2$-minors of
\begin{equation}
\label{eq:twobythree}
 \begin{pmatrix} x & y & z \\  
\psi_1(x,y,z) & \psi_2(x,y,z) & \psi_3(x,y,z) \\ 
\end{pmatrix} .
\end{equation}
The ideal generated by these minors is Cohen-Macaulay of codimension $2$
and its degree equals the cardinality of $Z$. This implies
that this ideal coincides with $I_Z$. The Hilbert-Burch Theorem ensures
that the transpose of (\ref{eq:twobythree}) equals $\varPhi_Z$.
Since $x,y,z$ are linearly independent, we see that $\varPhi_Z$ has the form
required in~(\ref{eq:asrequired}).

For the converse, suppose that the Hilbert-Burch matrix of $Z$ has size
$3 \times 2$ as in (\ref{eq:asrequired}) with $L_1,L_2,L_3$ linearly independent.
By performing row operations over $K$, we can replace $L_1,L_2,L_3$ by
$x,y,z$. This means that the transpose of $\varPhi_Z$ is (\ref{eq:twobythree})
for some $\psi_1,\psi_2,\psi_3$. Any such triple of ternary forms of degree $d-1$
arises from some tensor $A \in (K^3)^{\otimes d}$. By construction, $Z$ is the
eigenconfiguration of~$A$.
\end{proof}

Proposition \ref{prop:asrequired} translates into an algorithm for testing whether a given 
$Z \in (\mathbb{P}^2)^{d^2-d+1}$ is an eigenconfiguration. The algorithm
starts by computing the ideal
$$ I_Z \,\,\,\,= \bigcap_{(\alpha:\beta:\gamma) \in Z}  \!\! \bigl\langle\,
x \beta- y \alpha\,, \,
x \gamma - z \alpha\,, \,
y \gamma - z \beta \,\big\rangle. $$
This ideal must have three minimal generators of degree $d$; otherwise $Z$ is not
an eigenconfiguration. If $I_Z$ has three  generators, then we compute the
two syzygies. They must have degrees $1$ and $d-1$, so 
the minimal free resolution of $I_Z$ looks like
$$  
0 \,\rightarrow \,S(-d-1) \oplus S(-2d+1) \,
\stackrel{\varPhi}{\longrightarrow}  \,S(-d)^{\oplus 3} \rightarrow \,I_Z \,\rightarrow \,0. $$
At this point we examine the matrix $\varPhi$. If the linear entries $L_1,L_2,L_3$ in the left column
are linearly dependent, then $Z$ is not an eigenconfiguration. Otherwise we perform row operations
so that $\varPhi^T$ looks like (\ref{eq:twobythree}). The last step is to pick
a tensor $A \in (K^3)^{\otimes d}$ that gives rise to the ternary forms $\psi_1,\psi_2,\psi_3$
in the second row of $\varPhi^T$.

The remaining task is to find a geometric interpretation of the criterion
in Proposition \ref{prop:asrequired}. This was given for $d=3$ in the result 
whose proof we now present.

\begin{proof}[Proof of Theorem~\ref{thm:nosixonconic}]
Fix  a configuration $Z \in (\mathbb{P}^2)^7$. Our claim states that the
Hilbert-Burch matrix $\varPhi_Z$ has format $3 \times 2$ as in (\ref{eq:asrequired}), with
$L_1,L_2,L_3$ linearly independent, if and only if no six of the points in $Z$ lie on a conic.

We begin with the
 only-if direction.
Take $p \in Z$ such that $Z \backslash \{p\}$ lies on
a conic $C$ in $\mathbb{P}^2$. Fix linear forms $L_1$ and $L_2$ that cut out $p$.
The cubics $CL_1$ and $CL_2$ vanish on $Z$. By Proposition~\ref{prop:asrequired},
we have $I_Z = \langle CL_1, CL_2,F \rangle $ where $F$ is another cubic.
Since $F$ vanishes at $p$, there exist quadrics $Q_1$ and $Q_2$ such that
  $F = Q_2 L_1 - Q_1 L_2$. 
  The generators of the  ideal $I_Z$ are the $2 \times 2$-minors of 
  \[ \varPsi \,= \, \left(
\begin{array}{cc}
L_1 & Q_1  \\
L_2 & Q_2  \\
0 & C 
\end{array}
\right).
\]
This means that $\varPsi$ is a Hilbert-Burch matrix $\varPhi_Z$ for $Z$. 
However, by Proposition~\ref{prop:asrequired}, the left column in
any $\varPhi_Z$ must consist of linearly independent linear forms.
This is a contradiction, which completes the proof of the only-if direction. 

We now establish the
 if direction.
 Fix any
 configuration $Z \in (\mathbb{P}^2)^7$ 
 of seven points that do not lie
   on a conic.
We first prove that the minimal free resolution of $I_Z$ has the 
following form, 
where $c $ is either $0$ or $1$:
\begin{equation}
\label{eq:mfr}
0 \,\rightarrow\, S(-4)^{\oplus (c+1)}\oplus S(-5) \,\rightarrow \, S(-3)^{\oplus 3}\oplus S(-4)^{\oplus c} 
\,\rightarrow \,I_Z \,\rightarrow \,0,.
\end{equation}

By the Hilbert-Burch Theorem, the resolution of $I_Z$ equals
\[
0 \,\rightarrow\, \bigoplus_{i=1}^t S(-b_i) \,
\stackrel{\varPhi_Z}{\longrightarrow}
 \,\bigoplus _{i=1}^{t+1} 
S(-a_i) \,\rightarrow\, I_Z \,\rightarrow \, 0, 
\]
where $t, a_1, \ldots, a_{t+1}, b_1, \dots, b_t
\in \mathbb{N}$ with $a_1 \geq \cdots \geq a_{t+1}$ and $b_1 \geq \cdots \geq b_t$.
We abbreviate
$e_i = b_i-a_i$ and $f_i = b_i-a_{i+1}$ for  $i \in \{1, \dots, t\}$. These invariants satisfy
\begin{itemize}
 \item[(i)] $f_i \geq e_i, e_{i+1}$, 
 \item[(ii)] $e_i, f_i \geq 1$.  
\end{itemize} 
Furthermore, Eisenbud shows in  \cite[Proposition 3.8]{Eis} that
\begin{itemize}
 \item[(iii)] $a_i = \sum_{j=1}^{i-1} e_j + \sum_{j=i}^t f_j$. 
\end{itemize}
There exist $3$
linearly independent cubics that vanish on the seven points in $Z$.
By \cite[Corollary~3.9]{Eis}, the ideal $I_Z$
has either $3 $ or $4$ minimal generators, so $t \in \{2,3\}$.

Suppose $t=2$. Then $a_1 = a_2 = a_3 =3$, and it follows from (iii) that 
$$ f_1+f_2  =  3 \quad {\rm and} \quad e_1 + f_2 = 3. $$
So, by (i) and (ii), we obtain $e_1 = f_1 = 1$ and $f_2 = 2$. This implies $b_1=5$ and $b_2 =4$.  Therefore, $I_Z$ has a minimal free resolution of type~(\ref{eq:mfr}) with $c = 0$. 

Next, suppose $t=3$. Then $a_1 \geq a_2  = a_3 = a_4 = 3$. From (iii) we now get 
\[
\begin{array}{lllllll}
f_1& + & f_2 & + & f_3 & = & a_1 \\
e_1& + & f_2 & + & f_3 & = &  3\\
e_1& + & e_2 & + & f_3 & = & 3. 
\end{array}
\]
By (ii), $e_1 = e_2 = f_2 = f_3 = 1$. Therefore, $b_2 = a_2+e_2 = 4$ and $b_3 = a_2+f_2= 4$. 
Corollary 3.10 in~\cite{Eis} says that 
\[ \sum_{i\leq j}e_if_j \,=\, e_1(f_1+f_2+f_3)+e_2(f_2+f_3)+e_3f_3 \,=\, a_1+3 \,=\, \deg Z\, =\, 7. \]
Hence $a_1 =4$, $b_1 = 5$, and $I_Z$ has a minimal free resolution of type~(\ref{eq:mfr}) with $c = 1$. 

To complete the proof, we now assume that
no six points of $Z$ lie on a conic. In particular,
no conic contains $Z$, so the minimal free resolution of $I_Z$
equals (\ref{eq:mfr}), with $c \in \{0,1\}$.
Suppose that $c=1$. The Hilbert-Burch matrix must be 
\[
\varPhi_Z \,\,=\,\, 
\begin{small} \left(
\begin{array}{ccc}
L_{00} & L_{01}  & Q_0 \\
L_{10} & L_{11}  & Q_1 \\
L_{20} & L_{21}  & Q_2 \\
0 & 0 &  L 
\end{array}
\right) \end{small}
\]
with $L,L_{ij}$ are linear and $Q_k$ are quadrics.
Then $\,I_Z =  \bigl\langle LQ_0', LQ_1', LQ_2',Q \bigr\rangle$, where
\[
Q_0' = 
\left|
\begin{matrix}
L_{00} & L_{01} \\
L_{10} & L_{11} 
\end{matrix} \right|,
Q_1' =  \left| 
\begin{matrix}
L_{00} & L_{01} \\
L_{20} & L_{21} 
\end{matrix}
\right|, \ 
Q_2' =  \left| 
\begin{matrix}
L_{10} & L_{11} \\
L_{20} & L_{21} 
\end{matrix}
\right|, \ 
Q = \begin{small} \left|
\begin{matrix}
L_{00} & L_{01}  & Q_0 \\
L_{10} & L_{11}  & Q_1 \\
L_{20} & L_{21}  & Q_2 
\end{matrix}
\right|. \end{small}
\]

The ideal generated by $L$ and $Q$ contains $I_Z$. The intersection of 
the curves $\{L \! = \! 0\}$ and $\{Q \! = \! 0\}$ is contained in $Z$. Note 
that
these curves share no positive-dimensional component, since $Z$ is zero-dimensional. 
Thus  $\{L= Q=0\}$ consists of four points. Let $L'$ be a linear form vanishing on
 two of the three other points.  Then the conic
  $\{LL' = 0\}$ contains six points of $Z$, which contradicts our assumption.

 Hence, $c=0$. The resolution (\ref{eq:mfr}) tells us that
 the Hilbert-Burch matrix equals
 \[
\varPhi_Z \,\, =\,\, \left(
\begin{array}{cc}
L_1   & Q_1 \\
L_2   & Q_2 \\
L_3   & Q_3 \\
\end{array}
\right), 
\]
with linear forms $L_i$ and conics $Q_j$. If $L_1,L_2,L_3$ were linearly dependent
then we can take $L_3 = 0$. So, the conic $Q_3$ 
contains the six points in $Z \backslash \{L_1= L_2=0\}$.
Consequently, the linear forms $L_1,L_2,L_3$ must be linearly independent.
Proposition~\ref{prop:asrequired} now implies that $Z$ is the eigenconfiguration
of some $3 \times 3 \times 3$-tensor.
\end{proof}

After taking the Zariski closure, we have ${\rm Eig}_3 = (\mathbb{P}^2)^7$. 
We shall now discuss the subvariety ${\rm Eig}_{3,{\rm sym}} $
of those eigenconfigurations that come from symmetric tensors. Consider the
three quadrics in the second row of 
(\ref{eq:twobythree}). We write these as
$$ \begin{matrix}
\psi_1({\bf x}) &=&  a_1  x^2 + a_2  x y  + a_3 x z + a_4 y^2  + a_5 y z + a_6 z^2, \\
\psi_2({\bf x}) &=&  b_1  x^2 + b_2  x y  + b_3 x z + b_4 y^2  + b_5 y z + b_6 z^2,\\
\psi_3({\bf x}) &=&  c_1  x^2 + c_2  x y  + c_3 x z + c_4 y^2  + c_5 y z + c_6 z^2 .
\end{matrix}
$$
We shall characterize the case of symmetric tensors in terms of these coefficients.

\begin{proposition} \label{prop:fiveeqns}
The variety ${\rm Eig}_{3,{\rm sym}} $ is irreducible of dimension $9$ in $(\mathbb{P}^2)^7$.
An eigenconfiguration $Z$ comes from
a symmetric tensor $\phi$ as in (\ref{eq:rank1sym}) if and only~if
$$ a_5-b_3 =  b_3-c_2 = 
 2 a_4-2 a_6-b_2+c_3 = 
 2b_6-2 b_1-c_5+a_2 = 
 2 c_1-2c_4-a_3+b_5  = 0.$$
\end{proposition}

\begin{proof}
There exists a symmetric tensor with eigenconfiguration $Z$ if and only if
there exist a cubic $\phi$ and a 
linear form $L = u_1 x + u_2 y + u_3 z$
such that 
\begin{equation}
\label{eq:3lineqns}
\psi_1 + L x = \frac{\partial \phi}{\partial x} ,\,\,\,
\psi_2  + L y = \frac{\partial \phi}{\partial y}  ,\,\,\,
\psi_3 + L z = \frac{\partial \phi}{\partial z} .
\end{equation}
We eliminate the cubic $\phi$ from this system by 
taking crosswise partial derivatives:
$$ 
\frac{ \partial \psi_1 }{\partial y} \,+\, x \frac{\partial L }{\partial y} 
\, \,= \,\, 
\frac{ \partial \psi_2 }{\partial x} \,+\, y \frac{\partial L }{\partial x}\,,
\, \ldots\,,
\frac{ \partial \psi_2 }{\partial z} \,+\, y \frac{\partial L }{\partial z} 
\, \,= \,\, 
\frac{ \partial \psi_3 }{\partial y} \,+\, z \frac{\partial L }{\partial x}.
$$
This is a system of linear equations in the $21$ unknowns
$a_1,\ldots,a_6,b_1,\ldots,b_6,c_1,\ldots,$ $c_6, u_1,u_2,u_3$.
By eliminating the last three unknowns $u_1,u_2,u_3$ from that system, we arrive
at the five linearly independent equations in $a_i,b_j,c_k$ stated above.
\end{proof}

Proposition~\ref{prop:fiveeqns} translates into an algorithm
for testing whether a given configuration $Z$ is the eigenconfiguration
of a ternary cubic $\phi$. Namely, we compute the syzygies of $I_Z$,
we check that the Hilbert-Burch matrix has the  form
(\ref{eq:twobythree}), and then we check the five linear equations.
If these hold then $\phi$ is found by solving~(\ref{eq:3lineqns}).

While the equations in Proposition \ref{prop:fiveeqns} are linear,
we did not succeed in computing the prime ideal of 
${\rm Eig}_{3,{\rm sym}} $ in the homogeneous coordinate ring of
 $(\mathbb{P}^2)^7$. This is a challenging elimination problem.
 Some insight can be gained by intersecting
 ${\rm Eig}_{3,{\rm sym}}$ with natural  subfamilies of $(\mathbb{P}^2)^7$.
For instance, assume that $Z$ contains
the three coordinate points, so we restrict to the 
subspace $(\mathbb{P}^2)^4$ defined~by
$$ Z = \bigl\{(1 \!:\! 0 \!:\! 0),(0 \!:\! 1 \!: \! 0),(0 \!:\! 0\! : \! 1),
(\alpha_1\!:\alpha_2 \!: \! \alpha_3),
(\beta_1 \! : \! \beta_2 \! : \! \beta_3),
(\gamma_1 \! : \! \gamma_2 \!: \! \gamma_3),
(\delta_1 \! : \! \delta_2 \! : \! \delta_3) \bigr\}. $$
At this point it is important to recall that our problem is not
projectively invariant.

\begin{theorem}
The variety  ${\rm Eig}_{3,{\rm sym}} \cap  (\mathbb{P}^2)^4$
is three-dimensional, and it represents the 
eigenconfigurations $Z$ of the ternary cubics in the Hesse family
\begin{equation}
\label{eq:hessefamily}
 \phi \,=\, ax^3 + by^3 + c z^3 + 3d x y z . 
 \end{equation}
If $a,b,c,d$ are real then the eigenconfiguration $Z$ contains
at least five real points.
\end{theorem}

\begin{proof}
A ternary cubic $\phi = \sum_{i+j+k=3} c_{ijk} x^i y^j z^k $ has
$(1:0:0)$ as an eigenpoint of $\phi$ if and only if
$c_{210} = c_{201} = 0$. Likewise, $(0:1:0)$ 
is an eigenpoint if and only if $c_{120} = c_{021} = 0$,
and $(0:0:1)$ is an eigenpoint if and only if $c_{102} = c_{012} = 0$.
Hence the eigenconfiguration of $\phi$ contains all three coordinate points 
if and only if $\phi$ is in the Hesse family (\ref{eq:hessefamily}).
Since ${\rm Eig}_{3,{\rm sym}}$ has codimension $5$ in $(\mathbb{P}^2)^7$,
the intersection ${\rm Eig}_{3,{\rm sym}} \cap  (\mathbb{P}^2)^4$ has codimension
$\leq 5$, so its dimension is $\geq 3$.
The Hesse family is $3$-dimensional, and so we conclude that  
$\,{\rm dim}\bigl({\rm Eig}_{3,{\rm sym}} \cap  (\mathbb{P}^2)^4 \bigr) = 3$.

The four other eigenpoints of (\ref{eq:hessefamily}) are
 $\, \left(d\chi(\chi^2-1) : \chi(a\chi-c) : d(\chi^2-1)\right) $,
where $\chi$ runs over the zeros of the polynomial
\begin{equation}
\label{eq:quartic} 
d(a^2-d^2) \chi^4
\,-\, a(ab+cd) \chi^3
\,+\, 2(abc+d^3) \chi^2
\,-\, c(bc+ad) \chi
\,+ \, d(c^2-d^2).
\end{equation}
We claim that this quartic polynomial has at least two real roots for
all $a,b,c,d\in \mathbb{R}$.

Inside the projective space $\mathbb{P}^4$ of quartics
$\, f(\chi) =  k_4 \chi^4 + k_3 \chi^3 + k_2 \chi^2 + k_1 \chi + k_0$,
the family (\ref{eq:quartic}) is contained in the  hypersurface defined by the quadric
\[ \mathcal{S}\, =\, 2 k_4 k_2 + k_2^2 - 4 k_1 k_3 + 4 k_0 k_4 + 2 k_0 k_2 .\]
The discriminant of $f(\chi)$ defines a hypersurface of
degree $6$ in $\mathbb{P}^4$. One of the open regions in the complement
of the discriminant consists of quartics $f(\chi)$ with no real roots. 
In polynomial optimization 
(cf.~\cite[Lemma 3.3]{BPT})
one represents this region by a formula of the following form, where
$\kappa $ is a new indeterminate:
$$ f(\chi) \,\,= \,\,
\begin{pmatrix} \chi^2 & \chi & 1 \end{pmatrix} \cdot
\begin{pmatrix} 
k_0 & k_1/2 & \kappa \\
k_1 & k_2 - 2 \kappa & k_3/2 \\
\kappa & k_3/2  & k_4 \end{pmatrix}
\cdot \begin{pmatrix} \chi^2 \\ \chi \\ 1 \end{pmatrix},
$$
The symmetric $3 \times 3$-matrix is required to be
positive definite for some  $\kappa \in \mathbb{R}$.
The condition of being positive definite is
 expressed by the leading principal minors:
$$ \mathcal{P} = k_0\,,\,\,\,   \mathcal{Q} = 
{\rm det} \begin{pmatrix} 
k_0 & k_1/2 \\
k_1/2 & k_2 - 2\kappa \\
\end{pmatrix}\,,\,\,\, \mathcal{R} = {\rm det}
\begin{small}
\begin{pmatrix} 
\,k_0 & k_1/2 & \kappa \\
\,k_1/2 & k_2 - 2 \kappa & k_3/2 \\
\,\kappa & k_3/2  & k_4 \end{pmatrix} 
\end{small}. $$
It remains to be proved that there is no
solution $(k_0,k_1,k_2,k_3,k_4,\kappa) \in \mathbb{R}^6$ to
$$  \mathcal{P}  >  0, \,\, \mathcal{Q} > 0, \,\, \mathcal{R} > 0 \,\, \,\hbox{and} \,\, \mathcal{S} = 0. $$
We showed this by computing a {\em sum-of-squares proof}, in the sense of
 \cite[Chapter 3]{BPT}. More precisely, using the software {\tt SOSTools} \cite{SOS},
 we found explicit  polynomials $p,q,r,s \in \mathbb{R}[k_0,k_1,k_2,k_3,k_4,\kappa]$
 with floating point coefficients  such that 
 $$ p,q,r \,\,\,\hbox{are sums of squares} \,\,\,\,\, \hbox{and} \,\,\,\,\,
 p\mathcal{P} + q\mathcal{Q} + r\mathcal{R} + s\mathcal{S} = -1 . $$
 We are grateful to  Cynthia Vinzant for helping us with this computation.
 \end{proof}

We close this section by returning to tensors in
$(K^3)^{\otimes d}$ for general $d \geq 3$.

\begin{theorem}
Consider the spaces of eigenconfigurations of ternary tensors,
$$ {\rm Eig}_{d,{\rm sym}} \,\subset \,
{\rm Eig}_d \,\subset \, (\mathbb{P}^2)^{d^2-d+1} 
\qquad {\rm for} \,\, \, d \geq 3. $$
These projective varieties are irreducible, and their dimensions are
$$ {\rm dim}(\mathrm{Eig}_{d,{\rm sym}}) = \frac{1}{2}(d^2 + 3d)
\quad {\rm and} \quad
{\rm dim}(\mathrm{Eig}_d) = d^2+2d-1.
$$
\end{theorem}

\begin{proof}
First we show ${\rm dim}(\mathrm{Eig}_d) = d^2+2d-1$. 
Let $W$ be the set of $2 \times 3$ matrices (\ref{eq:twobythree}).
This is a $3 {d+1 \choose 2}$-dimensional vector space over $K$. 
The group
\[
G = \left\{\left. 
\left(
\begin{array}{cc}
1 & 0 \\
f & a 
\end{array}
\right)
\, \right| \, \mbox{$a \in K \backslash \{0\}$ and $f$ is a ternary form of degree $d-2$}
\right\}
\] 
acts on $W$ by left multiplication. Consider $\varphi, \omega \in W$.
It is immediate to see that if $\omega = g \cdot \varphi$ for some $g \in G$, then the variety defined by the $2\times 2$-minors of $\varphi$ equals the variety defined by the $2 \times 2$-minors of $\omega$. The converse also holds  because of the uniqueness of the Hilbert-Burch matrix.
The set $W^\circ$ of elements in $W$ whose $2 \times 3$-minors define  $d^2-d+1$ distinct points is an open subset of $W$. Therefore,
\begin{eqnarray*}
{\rm dim}(\mathrm{Eig}_d) \,\,=\,\, \dim W^\circ/G & = & \dim W^\circ - \dim G \\
 & = & 3 {d+1 \choose 2} - \left[{d \choose 2}+1\right] 
 \,\, = \,\, d^2+2d-1. 
\end{eqnarray*}

Next we prove ${\rm dim}(\mathrm{Eig}_{d,{\rm sym}}) = \frac{1}{2}(d^2 + 3d)$. 
We introduce the linear subspace
\[ U \, := \, \biggl\{\left. 
\left(
\begin{array}{ccc}
x & y & z \\
\partial \phi/\partial x & \partial \phi/\partial y  & \partial \phi/\partial z
\end{array}
\right) \, \right| \, \phi \,\,\hbox{ternary form of degree $d$} \bigg\} \,\subset \, W. 
\]
The action of the group $G$ on $W$ does not restrict to $U$. In fact, we notice that
\[
\left(
\begin{array}{cc}
1& 0 \\
f & a 
\end{array}
\right)
\left(
\begin{array}{ccc}
x & y & z \\
\partial \phi/\partial x & \partial \phi/\partial y  & \partial \phi/\partial z
\end{array}
\right) \in U
\] 
if and only if $f=0$. 
Let $U^\circ = U \cap W^\circ$ and consider the subgroup
\[
H = \left\{\left. 
\left(
\begin{array}{cc}
1 & 0 \\
0 & a 
\end{array}
\right)
\, \right| \, \mbox{$a \in K \backslash \{0\}$}
\right\} \,\,\subset \,\, G.
\]
This yields  $\,{\rm dim}(\mathrm{Eig}_{d,{\rm sym}}) = \dim U^\circ /H = 3 {d+2 \choose 2}- 1 = \frac{1}{2}(d^2 + 3d)$, 
as desired.  Our configuration spaces $\mathrm{Eig}_{d}$ and $\mathrm{Eig}_{d,{\rm sym}}$ are irreducible varieties
because they contain the irreducible varieties
$W^\circ/G$ and $U^{\circ}/H$ respectively as dense open subsets.
\end{proof}

\section{Real eigenvectors and dynamics}

In this section we focus on the real eigenpoints of
a tensor $A$ in  $(\mathbb{R}^{n})^{\otimes d}$.
If $A$ is generic then the number of eigenpoints in $\mathbb{P}^{n-1}_{\mathbb{C}}$
equals $((d-1)^n-1)/(d-2)$. Our hope is to show
that all of them lie in $\mathbb{P}^{n-1}_{\mathbb{R}}$
for suitably chosen symmetric tensors $\phi$.
A second question is how many of these real eigenpoints
are {\em robust}, in the sense that they are attracting fixed
points of the dynamical system $\nabla \phi : \mathbb{P}^{n-1}_{\mathbb R} 
\dashrightarrow  \mathbb{P}^{n-1}_{\mathbb R}$.
Our results will inform future
numerical work along the lines of \cite[Table 4.12]{CDN}.

We begin with a combinatorial construction for
 the planar case ($n=3$).
 Consider an arrangement $\mathcal{A}$ of $d$ distinct lines
in $\mathbb{P}_{\mathbb R}^2$, and let $\phi$ be the product 
of $d$ linear forms in $x,y,z$ that define the lines in $\mathcal{A}$.
We assume that $\mathcal{A}$ is generic in the sense  that no three
lines meet in a point. Equivalently,
the matroid of $\mathcal{A}$ is a uniform rank $3$ matroid on $d$ elements.
Such an arrangement $\mathcal{A}$ has $\binom{d}{2}$ {\em vertices}
in $\mathbb{P}_{\mathbb R}^2$, and these are the singular points
of the reducible curve $\{ \phi = 0\}$. The 
complement of $\mathcal{A}$  in
$\mathbb{P}_{\mathbb R}^2$ has
$ \binom{d}{2}+1$ connected components, called the
{\em regions} of~$\mathcal{A}$.

We are interested in the eigenconfiguration of $\mathcal{A}$, by which we
mean the eigenconfiguration of the symmetric tensor $\phi$.
Theorem \ref{thm:a} gives the expected number 
\begin{equation}
\label{eq:vertregi} \,
 d^2 - d + 1  \,\,=\,\,1+ (d{-}1) + (d{-}1)^2 \,\,= \,\,
 2 \binom{d}{2} + 1 \,\,= \,
 \hbox{\# vertices} + \hbox{\# regions} .
 \end{equation}
The following result shows that this is not just a numerical coincidence.

\begin{theorem} \label{thm:linesinP2}
A generic arrangement $\mathcal{A}$ 
of $d$ lines in $\mathbb{P}_{\mathbb R}^2$ has
$d^2-d+1$ complex eigenpoints and they are all real.
In addition to the $\binom{d}{2}$ vertices,
which are singular eigenpoints,
each of the $\binom{d}{2}+1$ regions of $\mathcal{A}$ contains precisely one
real eigenpoint.
\end{theorem}
 
\begin{proof}
The singular locus of the curve $\{\phi = 0\}$
consists of the vertices of the arrangement $\mathcal{A}$.
These are the eigenpoints with eigenvalue $0$. Their number is
$\binom{d}{2}$.

Let $L_1,L_2,\ldots,L_d$ be the linear forms that define the lines,
so $\phi = L_1 L_2 \cdots L_d$.
Consider the following optimization problem on the unit $2$-sphere:
$$ {\rm Maximize} \,\,
{\rm log}\, |\phi({\bf x})|\,\,=\,\,
\sum_{i=1}^d {\rm log}\,|L_i(x,y,z)| 
\quad \hbox{subject to}\,\,
x^2 + y^2 + z^2 = 1.
$$
The objective function takes the value $- \infty$ on the $d$ great circles
corresponding to $\mathcal{A}$. On each region of $\mathcal{A}$,
the objective function  
takes values in $\mathbb{R}$, and is
strictly concave. Hence there exists a unique local
maximum ${\bf u}^* = (x^*,y^*,z^*)$
in the interior of each region. Such a
maximum ${\bf u}^*$ is a critical point of 
the restriction of $\phi({\bf x})$ to the unit $2$-sphere.
The  Lagrange multiplier conditions state that
the vector ${\bf u}^*$ is parallel to the gradient of
$\phi$ at ${\bf u}^*$. This means that ${\bf u}^*$ is
an eigenvector of $\phi$, and hence the  pair $\pm {\bf u}^*$ 
defines a real eigenpoint of $\phi$ in the given region of $\mathbb{P}^2_{\mathbb{R}}$.

We proved that each of the $\binom{d}{2}+1$ regions of $\mathcal{A}$
contains one eigenpoint.
In addition, we have the $\binom{d}{2}$ vertices.
By Theorem \ref{thm:a},   the total
number of isolated complex eigenpoints cannot exceed $2 \binom{d}{2} + 1$.
This means that there are no
eigenpoints in $\mathbb{P}^2_{\mathbb{C}}$ other than those
already found. This completes the proof.
\end{proof}

We note that the line arrangement can be perturbed to a
situation where the map 
 $\nabla \phi : \mathbb{P}^{n-1}_{\mathbb R} 
\! \dashrightarrow \! \mathbb{P}^{n-1}_{\mathbb R}$ is regular,
i.e.~none of the eigenvectors has eigenvalue zero.

\begin{corollary}
There exists a smooth curve of degree $d$ in the real projective plane
$\mathbb{P}_{\mathbb R}^2$  whose complex eigenconfiguration consists of $d^2-d+1$ real points.
\end{corollary}

\begin{proof}
The eigenconfiguration of $\phi = L_1 L_2 \cdots  L_d\,$
is $0$-dimensional, reduced, and defined over $\mathbb{R}$.
By the Implicit Function Theorem,
these properties are preserved when $\phi$ gets perturbed
to a generic ternary form $\phi_\epsilon$ that is close to $\phi$.
\end{proof}

It is interesting to see what happens when the
matroid of $\mathcal{A}$ is not uniform. Here
the eigenconfiguration is not reduced. It arises from
Theorem \ref{thm:linesinP2} by degeneration.

\begin{example} \rm
Let $d=6$ and take $\mathcal{A}$ to be the line arrangement defined by
$$
\phi \,\, = \,\,  x\cdot y \cdot z \cdot (x-y) \cdot (x-z) \cdot (y-z).
$$
This is the reflection arrangement of type $A_4$. Its eigenscheme is non-reduced.
Each of the $12$ regions contains one eigenpoint as before, and
the simple vertices $(1:1:0)$,  $(1:0:1)$, and $(0:1:1)$ are eigenpoints
of multiplicity one. However, each of the triple points
$(1:0:0),(0:1:0),(0:0:1),(1:1:1)$ is an eigenpoint of multiplicity~$4$.
This makes sense geometrically:
in a nearby generic arrangement, such a vertex
splits into three vertices and one new region.
We note that the scheme structure 
at the eigenpoint $(1:0:0)$ is given by the primary ideal
$\langle 2yz-z^2, y^2-z^2 \rangle $.
\end{example}

The concavity argument concerning the optimization problem in the proof of
Theorem \ref{thm:linesinP2} works in arbitrary dimensions,
and we record this as a corollary.

\begin{corollary} \label{cor:numregions}
Each of the open regions of an arrangement $\mathcal{A}$ of $d$ hyperplanes in 
$\mathbb{P}^{n-1}_{\mathbb{R}}$  contains precisely one
real eigenpoint of $\mathcal{A}$. The number of  regions~is
\begin{equation}
\label{eq:numregions}
 \sum_{i=0}^{n-1} \binom{d-1}{i}. 
 \end{equation}
\end{corollary}

\begin{proof}
The first part has the same proof as the
one for $n=3$ given above.
The formula for the number of regions 
 can be found in \cite[Proposition 2.4]{Sta}.
\end{proof}

Theorem \ref{thm:linesinP2}
is restricted to $n=3$ because hyperplane arrangements
are singular in codimension $1$. Hence the eigenconfiguration
of a product of linear forms in $n \geq 4$ variables has
components of dimension $n-3$ in $\mathbb{P}^{n-1}_{\mathbb R}$.
We conjecture that a fully real eigenconfiguration 
can be constructed in the vicinity of such a tensor.

\begin{conjecture} \label{conj:optimistic}
Let $\phi$ be any product of $d$ nonzero linear forms in
${\rm Sym}_1(\mathbb{R}^n)$.
Every open neighborhood of $\phi$ in
${\rm Sym}_d(\mathbb{R}^n)$ contains
a symmetric tensor $\phi_\epsilon$ such that
all  $\,((d-1)^n-1)/(d-2)$ complex eigenpoints of $\,\phi_\epsilon$  are real.
\end{conjecture}

This optimistic conjecture is illustrated by the following
variant of Example~\ref{ex:cremona}.

\begin{example}[$n=d=4$] \rm
The classical Cremona transformation in $\mathbb{P}^3$ is
$\nabla \phi$ where $\phi = xyzw$ is the product of the coordinates.
The eigenconfiguration of $\phi$ consists of eight points, one for each sign
region in $\mathbb{P}^3$, and the six coordinate lines.
The expected number (\ref{eq:dustinnum}) of complex eigenpoints is $40$.
Consider the perturbation
$$ \begin{matrix}
 \phi_\epsilon\, = \,
 x y z w \,+\, \epsilon \bigl(
5 x^4 + 4 x^3 y - 2 x^2 y^2 - 8 x y^3 + 7 y^4 + 4 x^3 z + 2 x^2 y z + 2 x y^2 z \, \\
 \qquad \qquad + 2 y^3 z - 6  x^2 z^2 + 6 x y z^2 + 7 y^2 z^2 - 8 x z^3 + 3 y z^3 
 + 8 z^4 - 8 x^3 w + 2 x^2 y w  \\ 
 \qquad \qquad - 3 x y^2 w + 5 y^3 w + 8 x^2 z w 
 - 3 y^2 z w - 5 x z^2 w - 10 y z^2 w + 8 z^3 w - 5 x^2 w^2 \\ 
\qquad  \qquad - 6 x y w^2 
 - 3 y^2 w^2 - 6 x z w^2 + 3 y z w^2 + 3 x w^3 + 3 y w^3 - 4 z w^3 + 3 w^4 \bigr).
\end{matrix}
$$
All $40$ complex eigenpoints of this tensor are real,
so Conjecture~\ref{conj:optimistic} holds for $\phi$.
\end{example}

\begin{remark}
Conjecture \ref{conj:optimistic} is true for $n \leq 3$.
For $n = 3$ this follows from Theorem \ref{thm:linesinP2}:
 we can take $\phi_\epsilon $ to be any perturbation of the
given line arrangement~$\phi$.
For $n=2$ we take $\phi_\epsilon = \phi$ itself
because of the following fact:
if a binary form $\phi(x,y)$ is real-rooted
then also $x \frac{\partial  \phi}{\partial y}  - y \frac{\partial \phi}{\partial x}$ is real-rooted.
This follows from Corollary~\ref{cor:numregions}.
\end{remark}

We put the lid on this paper with a brief discussion of the 
dynamical system $\,\psi : \mathbb{P}^{n-1}_{\mathbb{R}} \dashrightarrow
 \mathbb{P}^{n-1}_{\mathbb{R}} \,$ associated with a tensor
 $A \in (\mathbb{R}^n)^{\otimes d}$. Iterating this map is known
 as the  {\em tensor power method}, and it is used as a tool
 in tensor decomposition \cite{AG}. This generalizes the
   power method of numerical linear algebra for computing
the eigenvectors of a matrix $A \in (\mathbb{R}^n)^{\otimes 2}$.
One starts with some unit vector ${\bf v}$ and repeatedly applies
the map ${\bf v} \mapsto \frac{A{\bf v}}{||A{\bf v}||}$. For generic inputs
 $A$ and ${\bf v}$, this iteration converges to the eigenvector corresponding to the largest
 absolute eigenvalue.

Suppose that ${\bf u} \in \mathbb{P}^{n-1}_{\mathbb{R}}$ is an eigenpoint
of a  given  tensor  $A \in (\mathbb{R}^n)^{\otimes d}$. We say that
${\bf u}$ is a {\em robust eigenpoint} if there exists an open neighborhood
$\mathcal{U}$ of ${\bf u}$ in $\mathbb{P}^{n-1}_{\mathbb{R}}$ such that,
for all starting vectors ${\bf v} \in \mathcal{U}$, the iteration
of the map $\psi$ converges to ${\bf u}$.

\begin{example}[Odeco Tensors] \rm
A symmetric tensor is {\em orthogonally decomposable} 
(this was abbreviated to {\em odeco} by Robeva \cite{Rob}) if it has the form
\[ \phi \,\,=\,\, \sum_{i = 1}^n a_i {\bf v}_i^{\otimes d} \]
where $a_1,\ldots,a_n\in \mathbb{R}$ and 
$\{{\bf v}_1,\ldots,{\bf v}_n\}$ is an orthogonal basis of  
$\mathbb{R}^n$. 
Following \cite{AG}, the robust eigenpoints of an odeco tensor are the
basis vectors ${\bf v}_i$, and they can be computed using the tensor power method.
Up to an appropriate change of coordinates, the odeco tensors are the
Fermat polynomials in (\ref{eq:fermatpoly}). The robust eigenpoints of
$\phi = x_1^d + \cdots + x_n^d$ are the coordinate
points ${\bf e}_1,\ldots,{\bf e}_n$. The region of
attraction of the $i^{\rm th}$ eigenpoint ${\bf e}_i$ under the
iteration of the map $\nabla \phi$ is the set of all points 
in $\mathbb{R}^{n-1}_{\mathbb{R}}$
whose $i^{\mathrm th}$ coordinate is largest in absolute value.
\end{example}

Odeco tensors for $n=d=3$
have three robust eigenvalues. At present we do not know any
ternary cubic $\phi$ with more than three robust eigenpoints.
Theorem~\ref{thm:linesinP2} might suggest that products
of linear forms are good candidates. However, we ran experiments 
with random triples of lines in $\mathbb{P}^2_{\mathbb{R}}$,
and we observed that the number 
of robust eigenvalues is usually one and occasionally zero.
We never found a factorizable ternary cubic $\phi$ with
two or more robust eigenpoints.
The Cremona map in Example \ref{ex:cremona} shows that $\phi = xyz$ is a cubic with zero
robust  eigenpoints.  Here is a similar example that points to the 
connection with frame theory in \cite{ORS}.

\begin{example}[$n=d=3$] \rm
We consider the factorizable ternary cubic
\begin{equation}
\label{eq:fradeco1}
 \phi \,\, = \,\, (2x+ 2 y  - z)(2x -y  + 2 z) (-x  + 2 y + 2 z) . 
 \end{equation}
This equals the  frame decomposable tensor
seen in \cite[Examples 1.1 and 5.2]{ORS}:
\begin{equation}
 \label{eq:fradeco2} \,\,
\phi \,= \, \frac{1}{24} \bigl(
   (-5 x+y+z)^3 \,+\,
(x-5y+z)^3 \,+\,
(x+y-5z)^3 \,+\,
(3x + 3y + 3z)^3 \bigr).
\end{equation}
Its gradient map $\mathbb{P}^2 \dashrightarrow \mathbb{P}^2$ is given by
$$
\nabla \phi \,\, = \,\, 3 \begin{pmatrix}
-4x^2+4xy+4xz+2y^2+yz+2z^2 \\
2x^2+4xy+xz-4y^2+4yz+2z^2 \\
2x^2+xy+4xz+2y^2+4yz-4z^2
\end{pmatrix}.
$$
This has four fixed points and three singular points,
for a total of seven eigenpoints:
$$  
\begin{matrix}
(1:1:-5),\,(1:-5:1),\,  (-5:1:1),\,  (3:3:3),  \\
(2:2:-1), \,\,(2:-1:2),\,\, (-1:2:2).
\end{matrix}
$$
Note that the pairwise intersections of the lines coincide with the coefficient vectors in (\ref{eq:fradeco1}).
By plugging $\nabla \phi$ into itself, we verify that the second iterate map equals
$$ (\nabla \phi)^{\circ 2} \,\, = \,\, \nabla \phi \circ \nabla \phi \,\,\,\, = \,\,\,\,
- 3^6 \phi(x,y,z) \cdot
\begin{pmatrix} \,x & y & z\, \end{pmatrix}^{\! T}.
$$
Hence $ (\nabla \phi)^{\circ 2}$ is the identity map
on all points in $\mathbb{P}^2_{\mathbb{R}} \backslash \{\phi = 0 \}$.
Every such point lies in a limit cycle of length two. 
The points on the curve $\{\phi = 0\}$ map to the singular points.
We conclude that the ternary cubic $\phi$ has no robust eigenpoints.
\end{example}

\bibliographystyle{amsalpha}

\end{document}